\documentclass{article}
\usepackage{amsmath,amssymb,mathrsfs,amsthm}
\usepackage{array,enumerate}
\usepackage{latexsym}
\usepackage{mathrsfs}
\usepackage{amscd}
\usepackage[all]{xy}
\usepackage{graphicx}
\theoremstyle{definition}
\newtheorem{thm}{Theorem}[section]
\newtheorem{dfn}[thm]{Definition}
\newtheorem{prop}[thm]{Proposition}
\newtheorem{cor}[thm]{Corollary}
\newtheorem{lem}[thm]{Lemma}
\newtheorem{rem}[thm]{Remark}
\newtheorem{prop-dfn}[thm]{Proposition-Definition}

\title{On a good reduction criterion for proper polycurves with sections}
\author{Ippei Nagamachi}

\begin{document}

\maketitle

\begin{abstract}
 We give a good reduction criterion for proper polycurves with sections, i.e.,
successive extensions of family of curves with section, under mild assumption.
This criterion is a higher dimensional version of the good reduction criterion for hyperbolic curves given by Oda and Tamagawa.
\end{abstract}

\tableofcontents

\section{Introduction}

 Let $K$ be a discrete valuation field with valuation ring $O _{K}$ and residue field $k$ of characteristic $p \geq 0$.
Let $K^{\rm sep}$ be the separable closure of $K$, $G_K := {\rm Gal}(K^{\rm sep}/K)$ the absolute Galois group of $K$, and $I_K$ its inertia subgroup. 
(Note that $I_K$, as a subgroup of $G_K$, 
depends on the choice of a prime ideal 
in the integral closure of $O_K$ in $K^{\rm sep}$ 
over the maximal ideal of $O_K$, but it is independent of 
this choice up to conjugation.)

When we are given a 
variety $X$ proper and smooth over $K$, 
it is an interesting problem to find a criterion 
for $X$ to admit good reduction, that is, to have 
a scheme $\mathfrak{X}$ proper and smooth over $O_K$ with generic fiber $X$.
(Such an $\mathfrak{X}$ is called a smooth model of $X$.) 

Generalizing results of N\'eron, Ogg, and Shafarevich for elliptic curves, 
Serre and Tate \cite{ST} proved that, when $X$ is an abelian variety over $K$, 
$X$ has good reduction if and only if the action of $I_K$ on 
the first $l$-adic etale cohomology $H^1(X \otimes K^{\rm sep}, \mathbb{Q}_l)$ 
is trivial, where $l$ is a prime not equal to $p$.

When $X$ is a proper 
hyperbolic curve (a geometrically connected 
proper smooth curve with genus $\geq 2$), 
it is not always true that $X$ has good reduction even if 
the action of $I_K$ on 
$H^1(X \otimes K^{\rm sep}, \mathbb{Q}_l)$ is trivial, namely, 
the first $l$-adic etale cohomology does not have enough information 
to know whether $X$ has good reduction or not.

If we consider the pro-$l$ completion 
$\pi_1(X \otimes K^{\rm sep}, \bar{t})^{l}$ 
of the etale fundamental group 
$\pi_1(X \otimes K^{\rm sep}, \bar{t})$ 
($\bar{t}$ is a geometric point of $X \otimes K^{\rm sep}$), 
it admits an outer action of $G_K$ 
(a continuous 
homomorphism 
$\rho: G_K \rightarrow {\rm Out}(\pi_1(X \otimes K^{\rm sep}, \bar{t})^{l}) 
:=  {\rm Aut}(\pi_1(X \otimes K^{\rm sep}, \bar{t})^{l})/ 
 {\rm Inn}(\pi_1(X \otimes K^{\rm sep}, \bar{t})^{l})$), 
thus the outer action $\rho|_{I_K}$ of $I_K$. 
This is expected to have a finer information than 
the action of $I_K$ on $H^1(X \otimes K^{\rm sep}, \mathbb{Q}_l)$ 
in certain cases. Actually, Oda \cite{Oda1} \cite{Oda2} proved that, for 
a proper hyperbolic curve $X$, $X$ has good reduction if and only if 
the outer action $\rho|_{I_K}$ is trivial. 
(More strongly, he proved that $X$ has good reduction 
if and only if the outer action of $I_K$ on 
$\pi_1(X \otimes K^{\rm sep}, \bar{t})^{l}/\Gamma_{n} 
\pi_1(X \otimes K^{\rm sep}, \bar{t})^{l}$ are trivial for any $n$, 
where $\{\Gamma_{n} 
\pi_1(X \otimes K^{\rm sep}, \bar{t})^{l}\}_n$ is 
the lower central filtration of 
$\pi_1(X \otimes K^{\rm sep}, \bar{t})^{l}$.)

Note that Oda's result is natural in the framework of 
anabelian geometry: In anabelian geometry, a hyperbolic curve 
is considered as a typical anabelian variety, that is, 
a variety which is determined by its outer Galois representation
$G_{K} \rightarrow \text{Out}\, \pi_1(X\otimes K^{\rm sep}, \bar{t})$ (under suitable assumption on $K$). 

The fact that a hyperbolic curve is anabelian in this sense, 
which is called the Grothendieck conjecture, 
is proven by Tamagawa \cite{Tama} and Mochizuki \cite{Moch1}, \cite{Moch2}. 
Therefore it would be natural to expect that, for an anabelian 
variety, a similar good reduction criterion to that of Oda will hold. 

Another class of varieties which are considered as anabelian 
is the class of proper hyperbolic polycurves, that is, 
varieties $X$ which admit a strucure of succesive smooth fibrations 
\begin{equation}\label{0}
X = X_n \overset{f_n}{\rightarrow} X_{n-1} \overset{f_{n-1}}{\rightarrow} \cdots 
\overset{f_2}{\rightarrow} X_1 
\overset{f_1}{\rightarrow} \mathrm{Spec}\,K 
\end{equation}
whose fibers are proper hyperbolic curves 
(we call such a structure a sequence of parameterizing morphisms): 
Indeed, the Grothendieck conjecture is known to hold for 
proper hyperbolic polycurves of dimension up to $4$ on suitable 
assumption on $K$, by Mochizuki \cite{Moch1} and Hoshi \cite{Ho}. 
Therefore it would be natural to consider on 
good reduction criterion for hyperbolic polycurves, 
which is the main interest in this paper.
For this good reduction criterion, we can also treat the case of genus 1
thanks to the criterion of N\'eron, Ogg, and Shafarevich.

If we admit the genus in the definition of proper hyperbolic curves to be $1$,
we say the resulting variety as a proper polycurve.
We call $X$ a proper polycurve with sections if it 
admits a sequence of parameterizing morphisms (1) 
such that each $f_i$ admits a section 
(we call such a structure a sequence of parameterizing 
morphisms with sections). 
When we fix a sequence of parameterizing 
morphisms with sections (1) of $X$, 
we call the maximum of the genera of fibers of $f_i$'s 
the maximal genus of (1), 
and when only $X$ is given, 
we call the minimum of the maximal genera of 
sequences of parameterizing 
morphisms with sections of $X$ the maximal genus of $X$.

Also, for such $X$ and any closed point $x$ of $X$, 
let $K(x)$ be the residue field of $x$ and consider the 
pro-$p'$ completion $\pi_{1}(X \otimes K(x)^{\rm sep}, \bar{x})^{p'}$ 
of the etale fundamental group, where $\bar{x}$ is a geometric point 
of $X \otimes K(x)^{\rm sep}$ above $x$. 
Because $x$ is $K(x)$-rational, 
 $\pi_1(X \otimes K(x)^{\rm sep}, \bar{x})^{p'}$ admits 
an action (not just an outer action) of the absolute Galois group $G_{K(x)}$ of $K(x)$.
Thus if we take a valuation ring $O_{K(x)}$ of $K(x)$ which contains $O_{K}$, we have the action of the inertia subgroup $I_{K(x)}$ 
on $\pi_{1}(X \otimes K(x)^{\rm sep}, \bar{x})^{p'}$. Then the main theorem is described as follows:

\begin{thm}
Let $K$ be as above and let $X$ be a proper polycurve 
with sections over $K$. Let $g$ be the maximal genus of $X$. 
Consider the following conditions: 
\begin{description}
\item{(A)}\mbox{} $X$ has good reduction. 
\item{(B)}\mbox{} For any closed point $x$ of $X$, and for any choice of valuation ring $O_{K(x)}$ of $K(x)$
as above, the action of $I_{K(x)}$ on $\pi_1(X \otimes K(x)^{\rm sep}, \bar{x})^{p'}$ 
is trivial.
\end{description}
Then (A) implies (B). If $p=0$ or $p > 2g+1$, (B) implies (A). 
\end{thm}

Since the implication (A) $\Rightarrow$ (B) is rather easy, we explain the strategy of the proof of the implication 
(B) $\Rightarrow$ (A)(assuming $p > 2g + 1$). Our proof heavily depends on 
the machinery of Tannakian categories.

For a prime number $l$ different from $p$ 
and a geometrically connected scheme $Y$ over a field $L$, 
let ${\rm Et}_l(Y \otimes L(y)^{\rm sep})$ be 
the category of 
smooth $\mathbb{Q}_{l}$-sheaves on $Y \otimes L(y)^{\rm sep}$, 
which is a Tannakian category over $\mathbb{Q}_l$. Here, $y$ is a closed point of $Y$ and $L(y)$ is the residue field of $y$.
For $r \in \mathbb{N}$, 
we define its Tannakian subcategories 
${\rm Et}^{\leq r}_l(Y \otimes L(y)^{\rm sep})$ 
(resp. $U{\rm Et}_l(Y \otimes L(y)^{\rm sep})$) 
as the minimal one which contains all the 
smooth $\mathbb{Q}_l$-sheaves of rank $\leq r$ 
(resp. the trivial smooth $\mathbb{Q}_l$-sheaf $\mathbb{Q}_l$) and 
which 
is closed under taking subquotients, tensor products, 
duals, and extensions. 
Also, for a geometrically connected morphism $f:Y \rightarrow Z$ of geometrically connected schemes over $L$, 
we define the Tannakian subcategory 
$U_f{\rm Et}^{\leq r}_l(Y \otimes L(y)^{\rm sep})$ of 
 ${\rm Et}_l(Y \otimes L(y)^{\rm sep})$ as the minimal one 
which contains the essential image of 
$f^*: {\rm Et}^{\leq r}_l(Z \otimes L(y)^{\rm sep}) \rightarrow 
{\rm Et}^{\leq r}_l(Y \otimes L(y)^{\rm sep})$ and which 
is closed under taking subquotients, tensor products, 
duals, and extensions. 
We denote the Tannaka dual of 
${\rm Et}^{\leq r}_l(Y \otimes L(y)^{\rm sep})$ 
(resp. $U{\rm Et}_l(Y \otimes L(y)^{\rm sep})$, 
$U_f{\rm Et}^{\leq r}_l(Y \otimes L(y)^{\rm sep})$) 
with respect to the fiber functor defined by 
a geometric point over $y$ by $\pi_1(Y \otimes L(y)^{\rm sep})^{l\text{-alg},r}$ 
(resp. $\pi_1(Y \otimes L(y)^{\rm sep})^{l\text{-unip}}$, 
$\pi_1(Y \otimes L(y)^{\rm sep})^{l\text{-rel-unip},r}$). 
(In the introduction, we omit to write the base point.) 
Note that these group schemes admit actions 
of the absolute Galois group $G_{L(y)}$ of $L(y)$. 

We take a sequence of parametrizing morphisms with sections 
(1) of $X$ whose maximal genus is equal to that of 
$X$, and prove the implication (B) $\Rightarrow$ (A) 
by induction on $n$. So we assume that 
$X_{n-1}$ has a good model $\mathfrak{X}_{n-1} \rightarrow \text{Spec}\,O_K$. 
The key ingredient of the proof is 
the homotopy exact sequence of Tannaka duals 
\begin{equation}\begin{split}
1 \rightarrow 
\pi_1((X \times_{X_{n-1}} x) \otimes K(x)^{\rm sep})^{l\text{-unip}}
& \rightarrow \pi_1(X \otimes K(x)^{\rm sep})^{l\text{-rel-unip},r} \\
& \rightarrow \pi_1(X_{n-1} \otimes K(x)^{\rm sep})^{l\text{-alg},r} \rightarrow 1, 
\end{split}\end{equation} 
where $x$ is any closed point of $X_{n-1}$, which is regarded also as a closed point of $X_{n}$ via the section of $f_{n}:X_{n} \rightarrow X_{n-1}$.  
This is an $l$-adic analogue of the 
homotopy exact sequences of de Rham and rigid fundamental groups 
of Lazda \cite{Laz}. Lazda's proof is motivic in some sense, and so 
his proof works also in our case without so much changes. 

We make a suitable choice of $l$ and prove by using the exact sequence 
(2) that the action of $I_{K(x)}$ on 
$\pi_1((X \times_{X_{n-1}} x) \otimes K(x)^{\rm sep})^{l\text{-unip}}$
is trivial for any closed point $x$ of $X$. 
(Here we use the assumption $p > 2g+1$.) 
On the other hand, 
we see from the relative theory of 
Tannakian category and (2) that 
$\pi_1((X \times_{X_{n-1}} x) \otimes K(x)^{\rm sep})^{l\text{-unip}}$'s 
naturally form a group scheme $\mathcal{E}$ over 
the category of smooth $\mathbb{Q}_l$-sheaves on $X_{n-1}$. 
The triviality of actions of $I_{K(x)}$'s on 
$\pi_1((X \times_{X_{n-1}} x) \otimes K(x)^{\rm sep})^{l\text{-unip}}$'s 
implies that the restriction of $\mathcal{E}$ to each 
$x$ is extendable to 
a group scheme over the category 
of smooth $\mathbb{Q}_l$-sheaves on the 
$O_{K(x)}$-valued point of $\mathcal{X}_{n-1}$ which extends $x$. 
This kind of property and a result of Drinfeld \cite{Dri}
imply that $\mathcal{E}$ is extendable to 
a group scheme over the category 
of smooth $\mathbb{Q}_l$-sheaves on $\mathcal{X}_{n-1}$. 
In particular, $\mathcal{E}$ is unramified at the generic point
$\xi$ of the special fiber of $\mathfrak{X}_{n-1}$. This and 
a variant of Oda's result imply that 
$X_n \rightarrow X_{n-1}$ has 
good reduction at the local ring $O_{\mathfrak{X}_{n-1},\xi}$
at $\xi$, and then a result of 
Moret-Bailly \cite{Mor} implies that the morphism $X_n \rightarrow X_{n-1}$ 
lifts to a smooth morphism $\mathfrak{X}_n \rightarrow \mathfrak{X}_{n-1}$, which 
implies (A).

The content of each section is as follows: 
In Section 2, we give a review and a preliminary result 
on $l$-unipotent envelope of profinite groups which we 
need in this paper. In Section 3, we give a review on 
Oda's good reduction criterion for proper hyperbolic curves 
and prove its variant, which uses 
$l$-unipotent envelope of etale fundamental groups. 
In Section 4, we prove a homotopy exact sequence of 
Tannaka duals of certain categories of 
smooth $\mathbb{Q}_l$-sheaves of the form (2). 
In Section 5, we give a review of Drinfeld's result 
on extension of smooth $\mathbb{Q}_l$-sheaves. We check that 
it is applicable in our situation, because the situation 
of Drinfeld is slightly more restrictive. 
In Section 6, we give a proof of the main theorem, 
using the results proved up to previous sections. 
In Section 7, we give a proof of Oda's good reduction criterion in \cite{Oda1} and \cite{Oda2}, which is not proved for a general discrete valuation field and is stated for general discrete valuation field without proof in \cite{Tama} Remark(5.4).

{\it Acknowledgements:} The author thanks his supervisor Atsushi Shiho for useful discussions and helpful advice.

\section{Review of $l$-unipotent envelope of profinite groups}
In this section, we recall basic facts on $l$-unipotent envelope of profinite groups.

 We start with a review of \cite{Del} \S 9. 
For an abstract group $G$, we denote the lower central series of $G$ as $ \{ \Gamma _{n}G \} _{n\geq 1}$.
We write the profinite (resp. pro-$p'$, resp. pro-$l$) completion of $G$ by $\hat{G}$ (resp. $G^{p'}$, resp. $G^{l}$),
where $l$ is a prime number and $p$ is a prime number or $0$. Here, the pro-$p'$ completion of $G$ is the limit of the projective
system of quotient groups of $G$ which are finite groups of order prime to $p$. Here, we consider that every finite group has order prime to $0$.
We also denote the lower central series of a profinite group $G$ as $ \{ \Gamma _{n}G \} _{n\geq 1} $,
where $\Gamma _{n}G$ is the closure of $\Gamma _{n}G$ (as abstract group) in $G$.
 For an abstract group $G$ and a prime number $l$, the natural morphism from $G$ to its pro-$l$ completion 
$G^{l}$ induces the isomorphism
 \begin{equation}
G^{l}/ \Gamma _{n}(G^{l}) \cong (G/ \Gamma _{n}G)^{l}
\end{equation}
for all $n\geq 1$, since both sides of this isomorphism are the limit of the projective system of quotient groups of $G$ 
which are finite $l$ groups and have nilpotent length $\leq n$.

\bigskip
\begin{prop-dfn}

The embedding functor
\begin{equation}
(\text{uniquely divisible nilpotent groups}) \rightarrow (\text{nilpotent groups})
\end{equation}
has a left adjoint functor (\cite{Bo} I\hspace{-1pt}I \S 4 ex15 and \S ex 6), which we denote by $G \mapsto G_{\mathbb{Q}}$.
We refer to $G_{\mathbb{Q}}$ as the {\it divisible closure} of $G$.
\label{2div}
\end{prop-dfn}

It is known that, when $G$ is a finitely generated torsion free nilpotent group, the adjunction morphism $G \rightarrow G_{\mathbb{Q}}$
is injective.

\bigskip
\begin{dfn}
\begin{enumerate}
\item  The {\it unipotent envelope} of an abstract finitely generated group $G$ is defined 
to be the Tannaka dual of the category of finite dimensional unipotent representations 
of $G$ over $\mathbb{Q}$. This is written by $G^{\mathrm{unip}}$.
\item The {\it $l$-unipotent envelope} of a finitely generated group (resp. a profinite group) $G$ is defined
to be the Tannaka dual of the category of finite dimensional unipotent representations (resp. finite dimensional continuous unipotent 
representations) of $G$ over $\mathbb{Q}_{l}$ .
This is written by $G^{l\text{-unip}}$.
\end{enumerate}
\end{dfn}

Let $N$ be a finitely generated torsion free  nilpotent group. Then it is known that we have the diagram 
\begin{equation}
N \hookrightarrow N_{\mathbb{Q}} \cong N^{\mathrm{unip}}(\mathbb{Q}),
\end{equation}
where the first map is the adjunction morphism defined by Proposition \ref{2div}.

\bigskip
On the other hand, for $N$ as above, the profinite completion $\hat{N}$ of $N$ is known to be isomorphic to the closure
of $N$ in $N^{\mathrm{unip}}(\mathbb{A}_{f}) := \underset{l}{\prod} ' N^{\mathrm{unip}}(\mathbb{Q}_{l})$. 
Since any finite nilpotent group is the product of their $l$-Sylow subgroups, we have
\begin{equation}
\hat{N} \cong \prod_{l} N^{l},
\end{equation}
where $l$ runs over all prime numbers. 
By looking at the $l$-component of the inclusion
\begin{eqnarray}
\prod_{l} N^{l} \cong \hat{N} \hookrightarrow N^{\text{unip}}(\mathbb{A}_{f}) = \underset{l}{\prod} ' N^{\mathrm{unip}}(\mathbb{Q}_{l}),
\end{eqnarray}
we obtain the inclusion
\begin{equation}
N^{l} \hookrightarrow N^{\mathrm{unip}}(\mathbb{Q}_{l}).
\end{equation}

Next, we recall the following fact on $l$-unipotent envelope of profinite groups. For more detailed explanation for $l$-unipotent envelope,
see \cite{Wil}.
\bigskip
\begin{prop} (\cite{Wil} Proposition 2.3)
Let $G$ be a finitely generated group, and $l$ be a prime number. Then, we have the isomorphism
\begin{equation}
G^{l\text{-unip}} \cong G^{\mathrm{unip}} \otimes_{\mathbb{Q}} \mathbb{Q}_{l}.
\end{equation} 
Moreover, since all the unipotent representations of $G$ over $\mathbb{Q}_{l}$ 
factor $G^{l}$, we have
\begin{equation}
(G^{l})^{l\text{-unip}} \overset{\sim}{\longleftarrow} G^{l\text{-unip}} \cong G^{\mathrm{unip}} \otimes_{\mathbb{Q}} \mathbb{Q}_{l}.
\end{equation}
\label{2lu}
\end{prop}
\bigskip

Let $\Sigma _{g}$ be the closed surface group of genus $g \geq 2$. Then, by the main theorem of \cite{Lab},
$\Gamma_{n}\Sigma _{g}/\Gamma_{n+1}\Sigma _{g}$ is a free abelian group for all $n\geq 1$.
It implies that $\Sigma _{g}/\Gamma_{n}\Sigma _{g}$ is a finitely generated torsion free nilpotent group.
Therefore, if we denote $(\Sigma _{g} )^{l}$ by $\pi$, we obtain the inclusion
\begin{eqnarray} \begin{split}
\pi / \Gamma _{n} \pi \cong (\Sigma _{g}/\Gamma_{n}\Sigma _{g})^{l} \hookrightarrow
(\Sigma _{g}/\Gamma_{n}\Sigma _{g})^{\mathrm{unip}}(\mathbb{Q}_{l}) 
&\cong ((\Sigma _{g}/\Gamma_{n}\Sigma _{g})^{\mathrm{unip}}
\otimes_{\mathbb{Q}} \mathbb{Q}_{l})(\mathbb{Q}_{l}) \\
& \cong ((\Sigma _{g}/\Gamma_{n}\Sigma _{g})^{l})^{l\text{-unip}}(\mathbb{Q}_{l}) \\
&\cong (\pi / \Gamma _{n} \pi )^{l\text{-unip}}(\mathbb{Q}_{l}).
\end{split}\label{eq:2long}\end{eqnarray}
We will use the inclusion (\ref{eq:2long}) in the next section.

\section{Good reduction criterion for proper hyperbolic curves with sections}

In this section, we recall the good reduction criterion for proper hyperbolic curves
proven by Oda and Tamagawa. Then, we give a modified form of it, when a given hyperbolic curve has a section.

\begin{dfn}
Let $S$ be a scheme and $X$ a scheme over $S$.
\begin{enumerate}
\item We shall say that $X$ is a {\it proper hyperbolic curve} (resp. {\it proper curve}) over $S$ if
the structure morphism $X \rightarrow S$ is smooth, proper, geometrically connected,
and of relative dimension one over $S$, each of whose geometric fiber is of genus $\geq 2$ (resp. $\geq1$). 
\item We shall say that $X$ is a {\it proper hyperbolic curve with a section} (resp.  {\it proper curve with a section}) over $S$ if
$X$ is a proper hyperbolic curve (resp. a proper curve) over $S$, and if the structure morphism has a section.
\end{enumerate}
\bigskip
Let $S$ be the spectrum of a discrete valuation ring $O_K$, $\eta$ the generic point of $S$,
$\sigma$ the closed point of $S$, $K = \kappa (\eta)$ the fractional field of $O_K$, $k = \kappa (\sigma)$ 
the residue field of $O_K$, and $p$ the characteristic of $k$.
\label{eq:3def}\end{dfn}

\begin{dfn}
Let $X \rightarrow \mathrm{Spec}\,K$ be a proper smooth morphism of schemes.
We say that $X$ has good reduction if there exists a proper smooth $S$-scheme $\mathfrak{X}$ 
whose generic fiber $\mathfrak{X} _{\eta}$ is isomorphic to $X$ over $K$. 
We refer to $\mathfrak{X}$ as a {\it smooth model} of $X$.
\bigskip
\end{dfn}

Let $X \rightarrow \mathrm{Spec}\,K$ be a proper hyperbolic curve.
Take a geometric point $ \overline{t} $ of $X \otimes K^{\mathrm{sep}}$. Then we have the exact sequence of profinite groups
\begin{eqnarray}
1 \rightarrow \pi _1 (X \otimes K^{\mathrm{sep}} ,\overline{t} ) \rightarrow
 \pi _1 (X,\overline{t} ) \rightarrow G_K \rightarrow 1.
\label{3fhes}
\end{eqnarray} 

This exact sequence yields the outer Galois action
\begin{eqnarray}
G_{K} \rightarrow \mathrm{Out}(\pi _1 (X \otimes K^{\mathrm{sep}},\overline{t})),
\end{eqnarray}
where, for a topological group $G$, $\mathrm{Out}(G)$ means the quotient group of the group $\mathrm{Aut}(G)$ of continuous group automorphisms of 
$G$ divided by the group $\mathrm{Inn}(G)$ of inner automorphisms of $G$.
Then, we have natural homomorphisms
\begin{eqnarray} \begin{split}
I_{K} \rightarrow G_{K} &\rightarrow \mathrm{Out}(\pi _1 (X \otimes K^{\mathrm{sep}},\overline{t})) \\
&\rightarrow \mathrm{Out}(\pi _1 (X \otimes K^{\mathrm{sep}},\overline{t})^{p'})
\rightarrow \mathrm{Out}(\pi _1 (X \otimes K^{\mathrm{sep}},\overline{t})^{l})
\end{split}\label{eq:3oa} \end{eqnarray}
for any prime number $l \neq p$.

\bigskip
Oda and Tamagawa gave the following criterion.

\begin{prop} (\cite{Oda1}\cite{Oda2}\cite{Tama} section 5)
The following are equivalent.
\begin{enumerate}
\item $X$ has good reduction.
\item The outer action $I_K \rightarrow \mathrm{Out}( \pi _1 (X \otimes K^{\mathrm{sep}} ,\overline{t}) ^{p'})$
defined by (\ref{eq:3oa}) is trivial
\item There exists a prime number $l \neq p$ such that the outer action
$I_K \rightarrow \mathrm{Out}( \pi _1 (X \otimes K^{\mathrm{sep}} ,\overline{t}) ^{l})$ defined by (\ref{eq:3oa}) is trivial.
\item  There exists a prime number $l \neq p$ such that the outer action of $I_K$ on 
$\pi _1 (X \otimes K^{\mathrm{sep}},\overline{t}) ^{l}/ \Gamma _{n}
\pi _1 (X \otimes K^{\mathrm{sep}},\overline{t}) ^{l}$ induced by (\ref{eq:3oa})
is trivial for all natural numbers $n$.
\end{enumerate}
\label{3odatama}
\end{prop}

In fact, this proposition is proved in \cite{Oda1} and \cite{Oda2} when the residue field of $K$ is of  characteristic $0$ and$K$ is a number field or a completion of a number field. In \cite{Tama} Remark 5.4, this proposition is stated for all discrete valuation field $K$ without proof.  Since, at the time of writing, a proof of this proposition seems to be not available in published form, we give a proof in section 7.  

Assume that the scheme $X$ is a proper curve over $\text{Spec}\,K$ and has a section $s : \text{Spec}\,K \rightarrow X$, and take a geometric point $\overline{s}$ over $s$.
Since we have the natural morphism from $\overline{s}$ to $X \otimes K^{\mathrm{sep}}$, we have the 
homotopy exact sequence \eqref{3fhes} with respect to the base point $\overline{s}$.
The section $s$ gives a section of the map $ \pi _1 (X,\overline{s}) \rightarrow G_K $ in the homotopy exact sequence.
This induces a homomorphism $G_K \rightarrow \mathrm{Aut}( \pi _1 (X \otimes K^{\mathrm{sep}} ,\overline{s}))$,
whose composition to $\mathrm{Out}(\pi _1 (X \otimes K^{\mathrm{sep}},\overline{s}))$ is same as the above homomorphism
$G_{K} \rightarrow \mathrm{Out}(\pi _1 (X \otimes K^{\mathrm{sep}},\overline{s}))$ in \eqref{eq:3oa}.
Therefore, for a prime number $l \neq p$, we obtain the following morphisms
\begin{equation}\begin{split}
I_{K} \hookrightarrow G_{K} &\rightarrow \mathrm{Aut}(\pi _1 (X \otimes K^{\mathrm{sep}} ,\overline{s})) \\
&\rightarrow \mathrm{Aut}(\pi _1 (X \otimes K^{\mathrm{sep}},\overline{s})^{p'}) \\
&\rightarrow \mathrm{Aut}(\pi _1 (X \otimes K^{\mathrm{sep}},\overline{s})^{l})
\rightarrow \mathrm{Aut}((\pi _1 (X \otimes K^{\mathrm{sep}},\overline{s})^{l})^{l\text{-unip}})
\end{split}\label{3doa}\end{equation}
by universal property of $l$-unipotent envelope. Here, the composition 
$\mathrm{Aut}(\pi _1 (X \otimes K^{\mathrm{sep}} ,\overline{s})) \rightarrow 
\mathrm{Aut}((\pi _1 (X \otimes K^{\mathrm{sep}},\overline{s})^{l})^{l\text{-unip}})$
can be identified with the morphism
\begin{equation}
\mathrm{Aut}(\pi _1(X \otimes K^{\mathrm{sep}} ,\overline{s}))
\rightarrow \mathrm{Aut}(\pi _1 (X \otimes K^{\mathrm{sep}},\overline{s})^{l\text{-unip}}),
\label{3cdoa}\end{equation}
which is also induced by universal property of $l$-unipotent envelope, via the isomorphism in Proposition \ref{2lu}.

\begin{prop}
The following are equivalent.
\begin{enumerate}
\item $X$ has good reduction.
\item The action of $I_K$ on $\pi _1 (X \otimes K^{\mathrm{sep}} ,\overline{s} )^{p'}$ defined by (\ref{3doa}) is trivial.
\item The action of $I_K$ on $\pi _1 (X \otimes K^{\mathrm{sep}} ,\overline{s} )^{l\text{-unip}}$ defined by (\ref{3doa}) and (\ref{3cdoa})
is trivial.
\end{enumerate}
\label{3gr}
\end{prop}

\begin{proof}
Assume that $X$ has good reduction, and let $\mathfrak{X}$ be a smooth model of $X$.
Fix a separable closure $k^{\mathrm{sep}}$ of $k$, the henselization $O^{\mathrm{h}}_{K}$,
and  the strict henselization $O^{\mathrm{sh}}_{K}$
of $O_{K}$ relative to $\mathrm{Spec}\,k^{\mathrm{sep}} \rightarrow \mathrm{Spec}\,k$.
Let $K^{\mathrm{sep}}$ be a separable closure of the fraction field of $O^{\mathrm{sh}}_{K}$. 
Then we have the following diagram.
\[
\xymatrix{
X \otimes _{K}K^{\mathrm{sep}} \ar[d] \ar[r] & X\otimes _{K}(\mathrm{Frac}\,O^{\mathrm{h}}_{K}) \ar[d] \ar[r] 
& \mathrm{Spec}\,(\mathrm{Frac}\,O^{\mathrm{h}}_{K}) \ar[d] \ar@<2.mm>[l] & \mathrm{Spec}\,K^{\mathrm{sep}} \ar[l] \ar[d] \\
\mathfrak{X}\otimes _{O_{K}}O^{\mathrm{sh}}_{K} \ar[r] & \mathfrak{X}\otimes _{O_{K}} O^{\mathrm{h}}_{K} \ar[r]
& \mathrm{Spec}\,O^{\mathrm{h}}_{K} \ar@<2.mm>@{.>}[l]
& \mathrm{Spec}\,O^{\mathrm{sh}}_{K} \ar[l]
}
\]

Since the morphism $\mathfrak{X}\otimes _{O_{K}} O^{\mathrm{h}}_{K} \rightarrow \mathrm{Spec}\,O^{\mathrm{h}}_{K}$ is proper, the unique section $s'$ of this morphism is induced by valuative criterion,
which is compatible with vertical arrows of the above diagram and the base change of the section $s$ by the morphism
$\text{Spec}\,(\mathrm{Frac}\,O^{\mathrm{h}}_{K}) \rightarrow \mathrm{Spec}\,K$.

Consider the etale fundamental groups of the schemes in the above diagram with the geometric points from the scheme $\text{Spec}\,K^{\mathrm{sep}}$ (denoted by $\bar{\eta}$).
Then, we have the following commutative diagram of homotopy exact sequences of 
profinite groups
\[
\xymatrix{ 
1 \ar[r] & \pi _{1}(X\otimes _{K} K^{\mathrm{sep}},\bar{\eta})\ar[d] \ar[r] 
& \pi _{1}(X\otimes _{K}(\mathrm{Frac}\,O^{\mathrm{h}}_{K}),\bar{\eta}) \ar[d] \ar[r]
& G_{\mathrm{Frac}\,O^{\mathrm{h}}_{K}} \ar[d] \ar@<2.mm>[l] \ar[r] & 1\\
1 \ar[r] & \pi _{1}(\mathfrak{X}\otimes _{O_{K}}O^{\mathrm{sh}}_{K},\bar{\eta}) \ar[r]
& \pi _{1}(\mathfrak{X}\otimes _{O_{K}} O^{\mathrm{h}}_{K},\bar{\eta}) \ar[r]
& \pi _{1}(\mathrm{Spec}\,O^{\mathrm{h}}_{K},\bar{\eta}) \ar@<2.mm>@{.>}[l] \ar[r] & 1.
}
\]
It holds that the first row is an exact sequence by \cite{SGA1} I\hspace{-1pt}X Theorem 6.1, and using the same argument in the proof of \cite{SGA1} I\hspace{-1pt}X Theorem 6.1, we can show that the second row is an exact sequence.
This diagram induces the commutative diagram of exact sequences of 
profinite groups
\[
\xymatrix{ 
1 \ar[r] & \pi _{1}(X\otimes _{K} K^{\mathrm{sep}},\bar{\eta})^{p'} \ar[d] \ar[r] 
& \pi _{1}(X\otimes _{K}(\mathrm{Frac}\,O^{\mathrm{h}}_{K}),\bar{\eta})^{(p')} \ar[d] \ar[r]
& G_{\mathrm{Frac}\,O^{\mathrm{h}}_{K}} \ar[d] \ar@<2.mm>[l] \ar[r] & 1\\
1 \ar[r] & \pi _{1}(\mathfrak{X}\otimes _{O_{K}}O^{\mathrm{sh}}_{K},\bar{\eta})^{p'} \ar[r]
& \pi _{1}(\mathfrak{X}\otimes _{O_{K}} O^{\mathrm{h}}_{K},\bar{\eta})^{(p')} \ar[r]
& \pi _{1}(\mathrm{Spec}\,O^{\mathrm{h}}_{K},\bar{\eta}) \ar@<2.mm>@{.>}[l] \ar[r] & 1.
}
\]
Here, the profinite group $\pi _{1}(X\otimes _{K}(\mathrm{Frac}\,O^{\mathrm{h}}_{K}),\bar{\eta})^{(p')}$ is the quotinet group $\pi _{1}(X\otimes _{K}(\mathrm{Frac}\,O^{\mathrm{h}}_{K}),\bar{\eta})/
\text{Ker}(\pi _{1}(X\otimes _{K} K^{\mathrm{sep}},\bar{\eta}) \rightarrow
\pi _{1}(X\otimes _{K} K^{\mathrm{sep}},\bar{\eta})^{p'}),$
and the profinite group $\pi _{1}(\mathfrak{X}\otimes _{O_{K}} O^{\mathrm{h}}_{K},\bar{\eta})^{(p')}$
is the quotient group $\pi _{1}(\mathfrak{X}\otimes _{O_{K}} O^{\mathrm{h}}_{K},\bar{\eta})/
\text{Ker}(\pi _{1}(\mathfrak{X}\otimes _{O_{K}}O^{\mathrm{sh}}_{K},\bar{\eta}) \rightarrow
\pi _{1}(\mathfrak{X}\otimes _{O_{K}}O^{\mathrm{sh}}_{K},\bar{\eta})^{p'}).$

Since the left vertical arrow is an isomorphism
and the action of $I_K$ on $\pi _{1}(\mathfrak{X}\otimes _{O_{K}}O^{\mathrm{sh}}_{K},\bar{\eta})^{p'}$
is trivial by the above diagram, the action of $I_K$ on
$\pi _{1}(X\otimes _{K} K^{\mathrm{sep}},\bar{\eta})^{p'}$
is also trivial.

\bigskip
 Assume that the action of $I_K$ on $\pi _1 (X \otimes K^{\mathrm{sep}},\overline{s})^{p'}$ is trivial.
Then, the action of $I_K$ on $\pi _1 (X \otimes K^{\mathrm{sep}},\overline{s})^{l\text{-unip}}$ is trivial by (\ref{3doa}) and (\ref{3cdoa}).

\bigskip
 Assume that the action of $I_K$ on $\pi _1 (X \otimes K^{\mathrm{sep}},\overline{s})^{l\mathrm{-unip}}$ is trivial.
By Proposition \ref{3odatama} or N\'eron-Ogg-Shafarevich criterion, it is sufficient to show that the action of $I_K$ on $\pi _1 (X \otimes K^{\mathrm{sep}},\overline{s} ) ^{l} / \Gamma _{n}
\pi _1 (X \otimes K^{\mathrm{sep}} ,\overline{s} ) ^{l}$ 
is trivial for all natural number $n$ in order to prove that $X$ has good reduction.

The action of $I_K$ on $\pi _1 (X \otimes K^{\mathrm{sep}} ,\overline{s})^{l\text{-unip}}$ is trivial, and we have 
the surjective morphism of affine group schemes

\begin{equation}
\pi _1 (X \otimes K^{\mathrm{sep}} ,\overline{s})^{l\text{-unip}} \twoheadrightarrow 
(\pi _1 (X \otimes K^{\mathrm{sep}} ,\overline{s}) / \Gamma _{n}
\pi _1 (X \otimes K^{\mathrm{sep}} ,\overline{s}) )^{l\text{-unip}}
\end{equation}
over $\mathbb{Q} _{l}$.
It follows that the homomorphism of their affine rings is injective, so the action of 
$I_K$ on 
$(\pi _1 (X \otimes K^{\mathrm{sep}} ,\overline{s}) / \Gamma _{n}
\pi _1 (X \otimes K^{\mathrm{sep}} ,\overline{s} )  )^{l\mathrm{-unip}}$
is trivial. 

 We have a natural injection (\ref{eq:2long}) in the previous section, by which the action of $I_K$ 
on $\pi _1 (X \otimes K^{\mathrm{sep}},\overline{s}) ^{l} / \Gamma _{n}
\pi _1 (X \otimes K^{\mathrm{sep}},\overline{s}) ^{l}$ 
is trivial for all natural number $n$. Therefore, $X$ has good reduction by Proposition \ref{3odatama}.
\end{proof}

\begin{rem}
In the above proposition, we only used the hypothesis that $X \rightarrow \text{Spec}\,K$ is proper, smooth, geometrically connected
and has a rational point, to prove $1 \Rightarrow 2$. In particular, we can show the following claim.

\bigskip
\noindent \textbf{Claim.}
Let $X$ be a proper smooth $K$-scheme with geometrically connected fibers, and $x$ be a closed point of $X$. Consider $K(x)$-scheme $X \otimes_{K} K(x)$
and the associated Galois action $I_{K(x)} \rightarrow \text{Aut}(\pi _1 ((X \otimes_{K} K(x)) \otimes_{K(x)} K(x)^{\mathrm{sep}} ,\overline{x})^{p'})$. Here, $\overline{x}$ is a geometric point over $\text{Spec}K(x)$.
If $X$ has good reduction, this action is trivial.

\label{3hgr}
\end{rem}

\section{Homotopy exact sequence of affine group schemes}

 In this section, we prove the existence of the homotopy exact sequence 
of affine group schemes which is similar to \cite{Wil} Corollary 3.2.
Wildeshaus showed it in the case of characteristic zero by using transcendental method, but we give an algebraic proof which works
also in positive characteristic case.
This exact sequence of affine group schemes plays a crucial role to prove the main theorem in this paper.
We obtain this exact sequence by applying the argument in \cite{Laz} 1.2  
to smooth $\mathbb{Q}_{l}$-sheaves instead of regular integrable connections.

\begin{dfn}
Let $r$ be a positive integer.
\begin{enumerate} 
\item Let $X$ be a connected Noetherian scheme and $l$ be a prime number invertible on $X$.
We denote the category of smooth $\mathbb{Q}_{l}$-sheaves on $X$ by $\text{Et}_{l}(X)$, which is a
Tannakian category over $\mathbb{Q}_{l}$.
Then, we define the its Tannakian subcategory $\text{Et}_{l}^{\leq r}(X)$ (resp. $U\text{Et}_{l}(X)$)
as the minimal one which contains all the smooth $\mathbb{Q}_{l}$-sheaves
of rank $\leq r$ (resp. the trivial smooth $\mathbb{Q}_{l}$-sheaf $\mathbb{Q}_{l}$) and which is closed under taking subquotients,
tensor products, duals, and extensions.
\item Let $f:X \rightarrow S$ be a proper smooth morphism between connected Noetherian schemes  and $l$ be a prime number invertible on $S$.
We define the Tannakian subcategory $U_{f}Et_{l}^{\leq r}(X)$ of $\text{Et}_{l}^{\leq r}(X)$ as the minimal one which contains the essential
image of $f^{*} : \text{Et}_{l}^{\leq r}(S) \rightarrow \text{Et}_{l}^{\leq r}(X)$ and which is closed under taking subquotients,
tensor products, duals, and extensions.
\item Let $f:X \rightarrow S$ be a proper smooth morphism between connected Noetherian schemes, $l$ be a prime number invertible on $S$,
and $s \rightarrow X$ be a geometric point.
We write the Tannaka dual of $\text{Et}_{l}^{\leq r}(X)$, (resp. $U\text{Et}_{l}(X)$, $U_{f}\text{Et}_{l}^{\leq r}(X)$)
with respect to the fiber functor defined by $s$ as $\pi _{1}(X,s)^{l\text{-alg},r}$
(resp. $\pi _{1}(X,s)^{l\text{-unip}}$, $\pi _{1}(X,s)^{l\text{-rel-unip},r}$).
\end{enumerate}
\label{3tannaka}
\end{dfn}

When $X$ is a proper smooth variety over a separably closed field, the category $U\text{Et}_{l}(X)$ is the same as $U_{f}\text{Et}_{l}^{\leq r}(X)$, where $f$ is the structure morphism. Thus, in this case the category $U\text{Et}_{l}(X)$ is a special case of the category $U_{f}\text{Et}_{l}^{\leq r}(X)$.

Let us recall some notions in the theory of Tannakian category.
We will denote the fundamental group of a Tannakian category $\mathcal{T}$ over a field $k$ by $\pi (\mathcal{T})$ (see [Del]\S 6). This is
an affine group scheme over $\mathcal{T}$, that is, a group object in the opposite category of 
the category of rings of $\text{Ind} \mathcal{T}$. Moreover, for $f:X \rightarrow S$ and $s \rightarrow X$ as in Definition \ref{3tannaka},
$s^{*}\pi (\text{Et}_{l}^{\leq r}(X)) \cong \pi _{1}(X,s)^{l\text{-alg},r}$
(resp. $s^{*}\pi (U\text{Et}_{l}(X)) \cong \pi _{1}(X,s)^{l\text{-unip},r}$,
 $s^{*}\pi (U_{f}\text{Et}_{l}^{\leq r}(X)) \cong \pi _{1}(X,s)^{l\text{-rel-unip},r}$).
\bigskip

Let $f:X \rightarrow S$ be a proper, smooth, and geometrically connected morphism with section $p$ between connected Noetherian schemes.
We fix a geometric point $s \rightarrow S$, and let $X_{s}$ be the base change of $X$ by $s \rightarrow S$.
We write the projection $X_{s} \rightarrow X$ by $i_{s}$, and the base change of $f$ and $p$ by $s \rightarrow S$ as $f'$ and $s'$.

We have functors of Tannakian categories
\begin{equation}
\text{Et}_{l}^{\leq r}(S) \overset{f^{*}}{\underset{p^{*}}{\rightleftarrows}} U_{f}\text{Et}_{l}^{\leq r}(X)
\overset{i_{s}^{*}}{\rightarrow} U\text{Et}_{l}(X_{s}),
\end{equation}
which induce homomorphisms
\begin{equation}
\pi _{1}(X_{s},s)^{l\text{-unip}} \overset{i_{s*}}{\rightarrow}
\pi _{1}(X,s)^{l\text{-rel-unip},r} \overset{f_{*}}{\underset{p_{*}}{\rightleftarrows}} \pi _{1}(S,s)^{l\text{-alg},r}
\end{equation}
between their Tannaka duals.

 Thanks to [Del], it can be seen that these morphisms of affine group schemes come from homomorphisms
between the fundamental groups
\begin{equation}
\pi (U_{f}\text{Et}_{l}^{\leq r}(X)) \overset{f_{*}}{\rightarrow} f^{*}\pi (\text{Et}_{l}^{\leq r}(S)),
\end{equation}
\begin{eqnarray}
p^{*}\pi (U_{f}\text{Et}_{l}^{\leq r}(X)) \overset{p_{*}}{\leftarrow} \pi (\text{Et}_{l}^{\leq r}(S)),
\label{eq:3reluni}
\end{eqnarray}
and
\begin{eqnarray}
\pi (U\text{Et}_{l}(X_{s})) \overset{i_{s*}}{\rightarrow} i_{s}^{*}\pi (U_{f}\text{Et}_{l}^{\leq r}(X)).
\end{eqnarray}

\begin{dfn} The {\it relatively $l$-unipotent fundamental group} of $X/S$ with respect to $(f,r)$ at
$p$ is defined to be the kernel of the morphism (\ref{eq:3reluni}).
This is an affine group scheme over $\text{Et}_{l}^{\leq r}(S)$.
We denote it by $\pi_{1}(X/S,r,p)^{\text{rel-}l\text{-unip}}$.
\end{dfn}

The morphisms of schemes $X_{s} \overset{i_{s}}{\rightarrow} X \overset{f}{\rightarrow} S$
induce homomorphisms
\begin{equation}
\pi (U\text{Et}_{l}(X_{s})) \overset{i_{s*}}{\rightarrow} i_{s}^{*}\pi (U_{f}\text{Et}_{l}^{\leq r}(X))
\overset{i_{s}^{*}f_{*}}{\rightarrow} i_{s}^{*}f^{*}\pi (\text{Et}_{l}^{\leq r}(S))
\end{equation}
of affine group schemes over $U\text{Et}_{l}(X_{s})$.
Taking fibers at $s$, we get
\begin{equation}
\pi _{1}(X_{s},s)^{l\text{-unip}} \overset{i_{s*}}{\rightarrow} \pi _{1}(X,s)^{l\text{-rel-unip},r}
\overset{f_{*}}{\rightarrow} \pi _{1}(S,s)^{l\text{-alg},r}.
\end{equation}
Since inverse image of objects in $\text{Et}_{l}^{\leq r}(S)$ by $f \circ i_{s} = f' \circ s$ is trivial,
the composition of the above morphisms is trivial.
Thus we have the unique morphism 
\begin{eqnarray}
\pi _{1}(X_{s},s)^{l\text{-unip}} = s^{*}\pi (U\text{Et}_{l}(X_{s})) \rightarrow s^{*}\pi _{1}(X/S,r,p)^{\text{rel-}l\text{-unip}}.
\label{eq:3lama}
\end{eqnarray}

The following theorem is an $l$-adic etale version of [Laz] Theorem 1.6.

\begin{thm}
Let us suppose that the rank of $R^{1}f_{*}\mathbb{Q}_{l}$ is $\leq r$.
Then, the morphism (\ref{eq:3lama}) is an isomorphism.
\label{4hes}
\end{thm}

Since $\pi _{1}(X,s)^{l\text{-rel-unip},r} \rightarrow \pi _{1}(S,s)^{l\text{-alg},r}$ is surjective,
this is equivalent to say that
\begin{equation}
1 \rightarrow \pi _{1}(X_{s},s)^{l\text{-unip}} \overset{i_{s*}}{\rightarrow} \pi _{1}(X,s)^{l\text{-rel-unip},r}
\overset{f_{*}}{\rightarrow} \pi _{1}(S,s)^{l\text{-alg},r} \rightarrow 1.
\end{equation}
is an exact sequence of affine group schemes over $\mathbb{Q}_{l}$.

We start the proof of Theorem \ref{4hes}, following the proof of Lazda given in [Laz] 1.2.
As in the proof of Lazda, it is sufficient to prove the following:
\begin{description}
\item{(A)}\mbox{}
If $\mathcal{E} \in U_{f}\text{Et}_{l}^{\leq r}(X)$ satisfies that $i_{s}^{*}\mathcal{E}$ is trivial,
then there exists $\mathcal{F} \in \text{Et}_{l}^{\leq r}(S)$ such that $\mathcal{E} \cong f^{*}\mathcal{F}$.
\item{(B)}\mbox{}
Let $\mathcal{E} \in U_{f}Et_{l}^{\leq r}(X)$, and let $\mathcal{F}_{0} \subset  i_{s}^{*}\mathcal{E}$ denote the largest
trivial sub-object.
Then there exists $\mathcal{E}_{0} \subset \mathcal{E}$ such that
$\mathcal{F}_{0} \cong i_{s}^{*}\mathcal{E}_{0}$ as a  sub-object of $i_{s}^{*}\mathcal{E}$.
\item{(C)}\mbox{}
For each $\mathcal{E} \in U\text{Et}_{l}(X_{s})$, there exists
$\mathcal{F} \in U_{f}\text{Et}_{l}^{\leq r}(X)$ and a surjective homomorphism $i_{s}^{*}\mathcal{F} \rightarrow \mathcal{E}$.
\end{description}

Before proving these assertions, we check that the restrictions of functors
$f_{*}, R^1f_{*} : U_{f}\text{Et}_{l}(X) \rightarrow \text{Et}_{l}(S)$ to
$U_{f}\text{Et}_{l}^{\leq r}(X) \rightarrow \text{Et}_{l}^{\leq r}(S)$ are well-defined.

\begin{dfn}
Let $g : Z \rightarrow W$ be a proper smooth and geometrically connected morphism between connected Noetherian schemes,
and $t$ be a natural number.
For objects $\mathcal{E} \in U_{g}\text{Et}_{l}^{\leq t}(Z)$, we define the notion of \textquotedblleft having
unipotent class $\leq m$ with respect to $(g,t)$\textquotedblright inductively as follows.
If $\mathcal{E}$ belongs to the essential image of $g^{*} : \text{Et}_{l}^{\leq t}(W) \rightarrow U_{g}\text{Et}_{l}^{\leq t}(W)$,
then we say $\mathcal{E}$ has unipotent class $\leq  1$ with respect to $(g,t)$. If there exists an extension
\begin{equation}
0 \rightarrow \mathcal{V} \rightarrow \mathcal{E} \rightarrow \mathcal{E}' \rightarrow 0
\end{equation}
with $\mathcal{E}'$ of unipotent class $\leq m-1$ and $\mathcal{V}$ of unipotent class $\leq 1$,
then we say that $\mathcal{E}$ has unipotent class $\leq m$.
\end{dfn}
\begin{lem}
The functors $f_{*}, R^{1}f_{*} : U_{f}\text{Et}_{l}^{\leq r}(X) \rightarrow \text{Et}_{l}^{\leq r}(S)$ are well-defined.
\end{lem}

\begin{proof} Let $\mathcal{E}$ be an element of $U_{f}\text{Et}_{l}^{\leq r}(X)$ whose unipotent class $\leq m$.
We use induction on $m$. For the case $m = 1$, there exists $\mathcal{F} \in \text{Et}_{l}^{\leq r}(S)$ such that $f^{*}\mathcal{F} \cong \mathcal{E}$.
Then, $f_{*}f^{*}\mathcal{F} \cong \mathcal{F} \in \text{Et}_{l}^{\leq r}(S)$ and
$R^{1}f_{*}f^{*}\mathcal{F} \cong R^{1}f_{*}\mathbb{Q}_{l}\otimes\mathcal{F} \in \text{Et}_{l}^{\leq r}(S)$ by projection formula.
For the case $m \geq 2$, we have an exact sequence
\begin{equation}
0 \rightarrow \mathcal{V} \rightarrow \mathcal{E} \rightarrow \mathcal{E}' \rightarrow 0
\end{equation}
with $\mathcal{E}'$ of unipotent class $\leq m-1$ and $\mathcal{V}$ of unipotent class $\leq 1$. Taking the long exact sequence
\begin{equation}
0 \rightarrow f_{*}\mathcal{V} \rightarrow f_{*}\mathcal{E} \rightarrow f_{*}\mathcal{E}' \rightarrow
R^{1}f_{*}\mathcal{V} \rightarrow R^{1}f_{*}\mathcal{E} \rightarrow R^{1}f_{*}\mathcal{E}',
\end{equation}
it follows that $f_{*}\mathcal{E}, R^{1}f_{*}\mathcal{E} \in \text{Et}_{l}^{\leq r}(S)$ by induction hypothesis.
\end{proof}

For $\mathcal{E} \in \text{Et}_{l}^{\leq r}(X)$ (resp. $\mathcal{E}' \in \text{Et}_{l}^{\leq r}(X_{s})$)
we denote the counit of the adjunction between $f^{*}$ and $f_{*}$ (resp. $f'^{*}$ and $f'_{*}$) as
$c_{\mathcal{E}}: f^{*}f_{*}\mathcal{E}  \rightarrow \mathcal{E}$
(resp. $c'_{\mathcal{E'}}: f'^{*}f'_{*}\mathcal{E'}  \rightarrow \mathcal{E'}$)

We first verify the assertion (A).
\begin{prop}
If $i_{s}^{*}\mathcal{E}$ is trivial,
then $c_{\mathcal{E}}: f^{*}f_{*}\mathcal{E}  \rightarrow \mathcal{E}$ is an isomorphism.
\label{4counit}
\end{prop}
\begin{proof} It is sufficient to show that the homomorphism:
\begin{equation}
i_{s}^{*}c_{\mathcal{E}}: i_{s}^{*}f^{*}f_{*}\mathcal{E} \rightarrow i_{s}^{*}\mathcal{E},
\end{equation}
which we get by pulling back $c_{\mathcal{E}}$ by $i_{s}$, is an isomorphism.
By proper base change theorem, 
\begin{equation}
i_{s}^{*}f^{*}f_{*}\mathcal{E} \cong f'^{*}s^{*}f_{*}\mathcal{E} \cong f'^{*}f'_{*}i_{s}^{*}\mathcal{E},
\end{equation}
so what we should show is that $c'_{\mathcal{E'}}$ is an isomorphism for any trivial $\mathcal{E'} \in \text{Et}(X_s)$.
This follows from the assumption that $f$ is geometrically connected.
\end{proof}
\bigskip

We next show the assertion (B).
\begin{prop}
Let $\mathcal{E} \in U_{f}\text{Et}_{l}^{\leq r}(X)$, and let $\mathcal{F}_{0} \subset  i_{s}^{*}\mathcal{E}$ denote the largest
trivial sub-object.
Then there exists $\mathcal{E}_{0} \subset \mathcal{E}$ such that
$\mathcal{F}_{0} \cong i_{s}^{*}\mathcal{E}_{0}$ as a  sub-object of $i_{s}^{*}\mathcal{E}$.
\end{prop}
\begin{proof}Let us denote $i_{s}^{*}\mathcal{E}$ as $\mathcal{F}$. We have a following commutative diagram
\[
\xymatrix{
\mathcal{F}_{0} \ar[r] & \mathcal{F} \\
 f'^{*}f'_{*}\mathcal{F}_{0} \ar[u]_{\wr c'_{\mathcal{F}_{0}}} \ar[r] &  f'^{*}f'_{*}\mathcal{F}. \ar[u]_{c'_{\mathcal{F}}} \ar@{.>}[lu]
}
\]
Since $\mathcal{F}_{0}$ is trivial, $c'_{\mathcal{F}_{0}}$ is an isomorphism,
which we have proved in the proof of proposition \ref{4counit}.
Since $f'^{*}f'_{*}\mathcal{F}$ is trivial, so is  the image of
$c'_{\mathcal{F}}: f'^{*}f'_{*}\mathcal{F}  \rightarrow \mathcal{F}$, and then we
get the unique homomorphism $f'^{*}f'_{*}\mathcal{F} \rightarrow \mathcal{F}_{0}$.
Hence $\mathcal{F}_{0}$ is the image of $f'^{*}f'_{*}\mathcal{F} \rightarrow \mathcal{F}$.

 By the proof of the previous proposition, $f'^{*}f'_{*}\mathcal{F} \rightarrow \mathcal{F}$
is obtained by pulling back $c_{\mathcal{E}}: f^{*}f_{*}\mathcal{E} \rightarrow \mathcal{E}$ by $i_{s}$.
Thanks to exactness of $i_{s}^{*}$, $\mathcal{F}_{0}$ is the inverse image of the image of $c_{\mathcal{E}}$.
\end{proof}
\bigskip

Finally, we start the proof of the assertion (C).

 For $n \in \mathbb{N}$, we define an object $\mathcal{U}_{n} \in U\text{Et}_{l}(X_{s})$ inductively as follows.
Let $\mathcal{U}_{1}$ be the trivial smooth $\mathbb{Q}_{l}$-sheaf of rank $1$ (denoted by $\mathbb{Q}_{l}$).
For $n \geq 1$, we will define $\mathcal{U}_{n}$ to be the extension of $\mathcal{U}_{n-1}$ by 
$f'^{*}(R^{1}f'_{*}(\mathcal{U}_{n-1}^{\vee}))^{\vee}$
corresponding to the identity under the isomorphisms:
\begin{equation} \begin{split} 
\mathrm{Ext}(\mathcal{U}_{n},f'^{*}(R^{1}f'_{*}(\mathcal{U}_{n-1}^{\vee}))^{\vee})
&\cong \mathrm{Ext}(\mathbb{Q}_{l},\mathcal{U}_{n}^{\vee}
\otimes f'^{*}(R^{1}f'_{*}(\mathcal{U}_{n-1}^{\vee}))^{\vee}) \\
&\cong H^{1}(X_{s},\mathcal{U}_{n}^{\vee} \otimes f'^{*}(R^{1}f'_{*}(\mathcal{U}_{n-1}^{\vee})^{\vee}))\\
&\cong R^{1}f'_{*}(\mathcal{U}_{n}^{\vee} \otimes f'^{*}(R^{1}f'_{*}(\mathcal{U}_{n-1}^{\vee})^{\vee}))\\
&\cong R^{1}f'_{*}(\mathcal{U}_{n}^{\vee}) \otimes (R^{1}f'_{*}(\mathcal{U}_{n-1}^{\vee})^{\vee})\\
&\cong \mathrm{End}(R^{1}f'_{*}(\mathcal{U}_{n}^{\vee})).
\end{split}\end{equation}
Taking higher direct images of the dual of the short exact sequence
\begin{equation}
0 \rightarrow f'^{*}(R^{1}f'_{*}(\mathcal{U}_{n-1}^{\vee}))^{\vee} \rightarrow \mathcal{U}_{n}
\rightarrow \mathcal{U}_{n-1} \rightarrow 0,
\end{equation}
we get the following exact sequence
\begin{equation}
0 \rightarrow f'_{*}(\mathcal{U}_{n-1}^{\vee}) \rightarrow f'_{*}(\mathcal{U}_{n}^{\vee})
\rightarrow R^{1}f'_{*}(\mathcal{U}_{n-1}^{\vee}) 
\overset{\delta}{\rightarrow} R^{1}f'_{*}(\mathcal{U}_{n-1}^{\vee})
\rightarrow R^{1}f'_{*}(\mathcal{U}_{n}^{\vee}).
\end{equation}

\begin{lem}
The connecting homomorphism $\delta$ is the identity.
\end{lem}
\begin{proof}
The element of
\begin{equation} 
\mathrm{Ext}(f'^{*}(R^{1}f'_{*}(\mathcal{U}_{n-1}^{\vee})),\mathcal{U}_{n}^{\vee})
\cong \mathrm{End}(R^{1}f'_{*}(\mathcal{U}_{n}^{\vee}))
\end{equation}
defined by the extension 
\begin{equation}
0 \rightarrow \mathcal{U}_{n-1}^{\vee} \rightarrow \mathcal{U}_{n}^{\vee}
\rightarrow f'^{*}(R^{1}f'_{*}(\mathcal{U}_{n-1}^{\vee})) \rightarrow 0
\end{equation}
is the identity.

From the fact that, for an extension
$0 \rightarrow \mathcal{E} \rightarrow \mathcal{F} \rightarrow f'^{*}\mathcal{V} \rightarrow 0$ of a trivial smooth
$\mathbb{Q}_{l}$-sheaf $f'^{*}\mathcal{V}$ by $\mathcal{E}$, the extension class under the isomorphism
\begin{equation}
\mathrm{Ext}(f'^{*}\mathcal{V},\mathcal{E}) \cong \mathcal{V}^{\vee}\otimes R^{1}f'_{*}(\mathcal{E})
\cong \mathrm{Hom}(\mathcal{V},R^{1}f'_{*}(\mathcal{E}))
\end{equation}
is nothing but the connecting homomorphism for the long exact sequence
\begin{equation}
0 \rightarrow f'_{*}(\mathcal{E}) \rightarrow f'_{*}(\mathcal{F}) \rightarrow \mathcal{V}
\rightarrow R^{1}f'_{*}(\mathcal{E}),
\end{equation}
the lemma follows.
\end{proof}

\bigskip
 In particular, any extension of $\mathcal{U}_{n-1}$ by a trivial smooth $\mathbb{Q}_{l}$-sheaf
$\mathcal{V}$ is split after pulling back to $\mathcal{U}_{n}$, and
$f'_{*}(\mathcal{U}_{n-1}^{\vee}) \cong f'_{*}(\mathcal{U}_{n}^{\vee})$. We get inductively the isomorphisms
$\mathbb{Q}_{l} \cong f'_{*}(\mathcal{U}_{1}^{\vee}) \cong f'_{*}(\mathcal{U}_{n}^{\vee})$ 
for all $n$.

Let $x = p'(s)$, $u_{1} = 1 \in (\mathcal{U}_{1})_{x} \cong \mathbb{Q}_{l}$, and choose an element $u_{n} \in (\mathcal{U}_{n})_{x}$ 
for $n$ inductively so that $(\mathcal{U}_{n})_{x} \rightarrow (\mathcal{U}_{n-1})_{x}$ sends
$u_{n}$ to $u_{n-1}$.

\begin{prop}
Let $\mathcal{F} \in U\text{Et}_{l}(X_{s})$ be an object of unipotent class $\leq m$ with respect to $(f',r)$ and $n \geq  m$.
Then for any $v \in \mathcal{F}_{x}$, there exists a morphism $\alpha : \mathcal{U}_{n} \rightarrow \mathcal{F}$
which send $u_{n}$ to $v$.
\end{prop}

\begin{proof}
We copy the proof of [Laz] Proposition 1.17. Let $\mathcal{F}$ be of unipotent class $\leq m$.
 To show the proposition, we use induction on $m$. The case $m = 1$ is trivial. For the case $m \geq 2$, choose an exact sequence
\begin{equation}
0 \rightarrow \mathcal{E} \overset{\psi}{\rightarrow} \mathcal{F}
\overset{\phi}{\rightarrow} \mathcal{G} \rightarrow 0,
\end{equation}
with $\mathcal{G}$ of unipotent class $\leq m-1$ and $\mathcal{E}$ of unipotent class $\leq 1$. By induction hypothesis, there exists a morphism
$\beta:\mathcal{U}_{n-1} \rightarrow \mathcal{G}$ such that
$\phi _{x}(v) = \beta _{x}(u_{n-1})$. Consider the following pull-back exact sequences of the above extension
with respect to $\mathcal{U}_{n} \overset{\text{nat}}{\rightarrow} \mathcal{U}_{n-1} \overset{\beta}{\rightarrow} \mathcal{G}$:

\[
\xymatrix{
0 \ar[r] & \mathcal{E} \ar[r] \ar@{=}[d] & \mathcal{F''} \ar[r] \ar[d] & \mathcal{U}_{n} \ar[d] \ar[r]
\ar@<2.mm>@{.>}[l] & 0 \\ 
0 \ar[r] & \mathcal{E} \ar[r] \ar@{=}[d] & \mathcal{F'} \ar[r] \ar[d] & \mathcal{U}_{n-1} \ar[d] \ar[r] & 0 \\ 
0 \ar[r] & \mathcal{E} \ar[r] & \mathcal{F} \ar[r] & \mathcal{G} \ar[r] & 0.
}
\]
As explained above, the extension of $\mathcal{U}_{n}$ by $\mathcal{E}$ split.
Fix a section $\mathcal{U}_{n} \rightarrow \mathcal{F''}$ and let us denote the induced morphism by
$\gamma : \mathcal{U}_{n} \rightarrow \mathcal{F}$. Then
$\phi _{x}(\gamma _{x}(u_{n}) - v) = 0$. By induction hypothesis, there exists
$\gamma ':\mathcal{U} _{n} \rightarrow \mathcal{E}$ such that $\gamma_{x} '(u_{n}) = \gamma _{x}(u_{n}) - v$.
Then, $\gamma - \psi \circ \gamma '$ satisfies the condition required for $\alpha$.
\end{proof}

\begin{cor}
For all $\mathcal{E} \in U\text{Et}_{l}(X_{s})$, there exists a surjective homomorphism
$\mathcal{U}_{m}^{\oplus N} \rightarrow \mathcal{E}$ for some $m,N \in \mathbb{N}$.
\label{4canounip}\end{cor}

\begin{proof} The assertion follows immediately from the proposition if we take a basis for $\mathcal{E}_{x}$.
\end{proof}

\bigskip
We will define $\mathcal{W}_{n} \in U_{f}\text{Et}_{l}^{\leq r}(X)$ whose restriction to $X_{s}$ is isomorphic to $\mathcal{U}_{n}$ inductively.
Moreover, $f_{*} \mathcal{W}_{n}$ will be defined so that $f_{*}(\mathcal{W}_n^{\vee}) \cong \mathbb{Q}_{l}$
and there exists a morphism $\varepsilon_{n} : p^{*}\mathcal{W}_{n}^{\vee} \rightarrow \mathbb{Q}_{l}$
such that the composite
$\mathbb{Q}_{l} \cong f_{*}(\mathcal{W}_{n}^{\vee}) \cong p^{*}f^{*}f_{*}(\mathcal{W}_{n}^{\vee})
\rightarrow p^{*}\mathcal{W}_{n}^{\vee} \overset{\varepsilon_{n}}{\rightarrow} \mathbb{Q}_{l}$
is an isomorphism. 

 We start the induction with $\mathcal{W}_{1} = \mathbb{Q}_{l}$. Let us assume that $\mathcal{W}_{n}$ is defined.
Then, we will define $\mathcal{W}_{n+1}$ to be an extension of $\mathcal{W}_{n}$
by the sheaf $f^{*}R^{1}f_{*}(\mathcal{W}_{n}^{\vee})^{\vee}$ so that the inverse image of the exact sequence
\begin{equation}
0 \rightarrow f^{*}(R^{1}f_{*}(\mathcal{W}_{n}^{\vee}))^{\vee} \rightarrow \mathcal{W}_{n+1}
\rightarrow \mathcal{W}_{n} \rightarrow 0
\end{equation}
to $X_{s}$ is isomorphic to
\begin{equation}
0 \rightarrow f'^{*}(R^{1}f'_{*}(\mathcal{U}_{n}^{\vee}))^{\vee} \rightarrow \mathcal{U}_{n+1}
\rightarrow \mathcal{U}_{n} \rightarrow 0.
\end{equation}
Now consider the extension group
\begin{eqnarray}
\mathrm{Ext}(\mathcal{W}_{n},f^{*}(R^{1}f_{*}(\mathcal{W}_{n}^{\vee}))^{\vee})
\cong H^{1}(X,\mathcal{W}_{n}^{\vee} \otimes f^{*}(R^{1}f_{*}(\mathcal{W}_{n}^{\vee}))^{\vee}).
\label{eq:4ec}\end{eqnarray}
 Let us denote $\mathcal{W}_{n}^{\vee} \otimes f^{*}(R^{1}f_{*}(\mathcal{W}_{n}^{\vee}))^{\vee}$ as $\mathcal{E}$.
The Leray spectral sequence for $\mathcal{E}$ associated to $f: X \rightarrow S$ gives us the 5-term exact sequence
\begin{equation}
0 \rightarrow H^{1}(S,f_{*}(\mathcal{E})) \rightarrow H^{1}(X,\mathcal{E})
\rightarrow H^{0}(S,R^{1}f_{*}(\mathcal{E})) \rightarrow H^{2}(S,f_{*}(\mathcal{E}))
\rightarrow H^{2}(X,\mathcal{E}).
\end{equation}
After rewriting the objects in the above exact sequence by projection formulas,
the isomorphism (\ref{eq:4ec}) and induction hypothesis, we obtain the following exact sequence
 \begin{eqnarray}\begin{split}
0 
&\rightarrow H^{1}(S,(R^{1}f_{*}(\mathcal{W}_{n}^{\vee}))^{\vee}) 
\rightarrow \mathrm{Ext}(\mathcal{W}_{n},f^{*}(R^{1}f_{*}(\mathcal{W}_{n}^{\vee}))^{\vee}) \\
&\rightarrow \mathrm{End}(R^{1}f_{*}(\mathcal{W}_{n}^{\vee})) 
\rightarrow H^{2}(S,R^{1}f_{*}(\mathcal{W}_{n}^{\vee})^{\vee}) \\
&\rightarrow H^{2}(X,\mathcal{W}_{n}^{\vee} \otimes (R^{1}f_{*}(\mathcal{W}_{n}^{\vee}))^{\vee}).
\end{split}\label{eq:45t}\end{eqnarray}

The isomorphism
$\mathbb{Q}_{l} \cong f_{*}(\mathcal{W}_{n}^{\vee}) \cong p^{*}f^{*}f_{*}(\mathcal{W}_{n}^{\vee})
\rightarrow p^{*}\mathcal{W}_{n}^{\vee} \overset{\varepsilon_{n}}{\rightarrow} \mathbb{Q}_{l}$
induces an isomorphism

\begin{eqnarray}\begin{split}
H^{i}(S,(R^{1}f_{*}(\mathcal{W}_{n}^{\vee}))^{\vee})
&\cong H^{i}(S, f_{*}(\mathcal{W}_{n}^{\vee}) \otimes (R^{1}f_{*}(\mathcal{W}_{n}^{\vee}))^{\vee}) \\
&\rightarrow H^{i}(S,p^{*}\mathcal{W}_{n}^{\vee} \otimes (R^{1}f_{*}(\mathcal{W}_{n}^{\vee}))^{\vee})
\rightarrow H^{i}(S,(R^{1}f_{*}(\mathcal{W}_{n}^{\vee}))^{\vee}).
\end{split}\label{eq:4rt1} \end{eqnarray}

Since this is the same as the following isomorphism

\begin{eqnarray}\begin{split}
H^{i}(S,(R^{1}f_{*}(\mathcal{W}_{n}^{\vee}))^{\vee})
&\rightarrow H^{i}(X, \mathcal{W}_{n}^{\vee} \otimes f^{*}(R^{1}f_{*}(\mathcal{W}_{n}^{\vee}))^{\vee}) \\
&\rightarrow H^{i}(S,p^{*}\mathcal{W}_{n}^{\vee} \otimes (R^{1}f_{*}(\mathcal{W}_{n}^{\vee}))^{\vee})
\rightarrow H^{i}(S,(R^{1}f_{*}(\mathcal{W}_{n}^{\vee}))^{\vee}),
\end{split}\label{eq:4rt2}\end{eqnarray}
the homomorphisms
$$H^{1}(S,(R^{1}f_{*}(\mathcal{W}_{n}^{\vee}))^{\vee}) 
\rightarrow \mathrm{Ext}(\mathcal{W}_{n},f^{*}(R^{1}f_{*}(\mathcal{W}_{n}^{\vee}))^{\vee}) $$
$$H^{2}(S,(R^{1}f_{*}(\mathcal{W}_{n}^{\vee}))^{\vee}) 
\rightarrow H^{2}(X,\mathcal{W}_{n}^{\vee} \otimes (R^{1}f_{*}(\mathcal{W}_{n}^{\vee}))^{\vee})$$
in the exact sequence (\ref{eq:45t}) split. 
Therefore the morphism $\mathrm{Ext}(\mathcal{W}_{n},f^{*}(R^{1}f_{*}(\mathcal{W}_{n}^{\vee}))^{\vee}) 
\rightarrow \mathrm{End}(R^{1}f_{*}(\mathcal{W}_{n}^{\vee}))$ have a unique section corresponding to the retraction, so
in the commutative diagram
\[
\xymatrix{
\mathrm{Ext}(\mathcal{W}_{n},f^{*}(R^{1}f_{*}(\mathcal{W}_{n}^{\vee}))^{\vee}) \ar[r] \ar[d]
& \mathrm{Ext}(\mathcal{U}_{n},f'^{*}(R^{1}f'_{*}(\mathcal{U}_{n}^{\vee}))^{\vee}) \ar@{=}[d] \\ 
\mathrm{End}(R^{1}f_{*}(\mathcal{W}_{n}^{\vee})) \ar[r]
& \mathrm{End}(R^{1}f'_{*}(\mathcal{U}_{n}^{\vee})),
}
\]
 the element $\text{id} \in \mathrm{End}(R^{1}f_{*}(\mathcal{W}_{n}^{\vee}))$
canonically lifts to $\mathrm{Ext}(\mathcal{W}_{n},f^{*}(R^{1}f_{*}(\mathcal{W}_{n}^{\vee}))^{\vee})$
by the section.
Then, $\mathcal{W}_{n+1}$ is defined to be the extension of $\mathcal{W}_{n}$
by $f^{*}(R^{1}f_{*}(\mathcal{W}_{n}^{\vee}))^{\vee}$ corresponding to this element.
Since this is sent to $\text{id} \in \mathrm{End}(R^{1}f'_{*}(\mathcal{U}_{n}^{\vee}))$, which corresponds to
the extension class of $\mathcal{U}_{n+1}$, we have natural isomorphism
$i_{s}^{*}\mathcal{W}_{n+1} \cong \mathcal{U}_{n+1}$. 

 To complete the induction, it is sufficient to show that $f_{*}(\mathcal{W}_{n+1}^{\vee}) \cong f_{*}(\mathcal{W}_{n}^{\vee})$
and that there exists a morphism $p^{*}\mathcal{W}_{n+1}^{\vee} \rightarrow \mathbb{Q}_{l}$
as in the induction hypothesis. By taking fibers at $s$ and proper base change theorem, we can prove the first claim.
For the second, we consider the following exact sequences
\[
\xymatrix{
0 \ar[r] & p^{*}\mathcal{W}_{n}^{\vee} \ar[r] \ar[d] & p^{*}\mathcal{W}_{n+1}^{\vee} \ar[r] \ar[d]
& R^{1}f_{*}(\mathcal{W}_{n}^{\vee}) \ar@{=}[d] \ar[r] & 0 \\ 
0 \ar[r] & \mathbb{Q}_{l} \ar[r] & (\text{pushout}) \ar[r] \ar@<2.mm>@{.>}[l]
& R^{1}f_{*}(\mathcal{W}_{n}^{\vee}) \ar[r] & 0,
}
\]
where the left vertical arrow is $\varepsilon _{n}$. Then, the lower exact sequence splits since the following diagram
\[
\xymatrix{
\mathrm{Ext}(R^{1}f_{*}(\mathcal{W}_{n}^{\vee})),p^{*}\mathcal{W}_{n}^{\vee}) \ar[d]
& \mathrm{Ext}(\mathcal{W}_{n},f^{*}(R^{1}f_{*}(\mathcal{W}_{n}^{\vee})^{\vee}) \ar[l] \ar[d] \\ 
\mathrm{Ext}(R^{1}f_{*}(\mathcal{W}_{n}^{\vee})),\mathbb{Q}_{l}) & H^{1}(S,R^{1}f_{*}(\mathcal{W}_{n}^{\vee})) \ar@{=}[l] 
}
\]
commutes and the right vertical arrow send the extension class defined by $\mathcal{W}_{n+1}$ to $0$.
Fixing its retraction, we have a homomorphism $p^{*}\mathcal{W}^{\vee}_{n+1} \rightarrow \mathbb{Q}_{l}$
such that the composition $p^{*}\mathcal{W}^{\vee}_{n} \rightarrow p^{*}\mathcal{W}^{\vee}_{n+1} \rightarrow \mathbb{Q}_{l}$ is equal to $\varepsilon _{n}$.
Now the second assertion follows immediately from the fact that the diagram
\[
\xymatrix{
 & f_{*}\mathcal{W}_{n+1}^{\vee} \ar[r] & p^{*}f^{*}f_{*}\mathcal{W}_{n+1}^{\vee} \ar[r] & p^{*}\mathcal{W}_{n+1}^{\vee} \ar[rd] \\ 
\mathbb{Q}_{l} \ar[r] \ar[ru] & f_{*}\mathcal{W}_{n}^{\vee} \ar[r] \ar[u] & p^{*}f^{*}f_{*}\mathcal{W}_{n}^{\vee} \ar[r] \ar[u] 
& p^{*}\mathcal{W}_{n}^{\vee} \ar[r] \ar[u] & \mathbb{Q}_{l}
}
\]
commutes.

 These arguments and Corollary \ref{4canounip} show the following proposition, which is the assertion (C).
\begin{prop}
For all $\mathcal{F} \in U\text{Et}_{l}(X_{s})$, there exists $\mathcal{E} \in U_{f}\text{Et}_{l}^{\leq r}(X)$ and a surjective homomorphism
$i_{s}^{*}\mathcal{E} \rightarrow \mathcal{F}$.
\end{prop}

\section{Extension of smooth $\mathbb{Q}_{l}$-sheaves}

In this section, we prove the results similar to \cite{Dri} 5.1 when the base scheme is a discrete valuation ring.
In \cite{Dri}, they are proved when the base scheme is $\mathbb{Z}$ and we check that they remain valid also in our situation.

 Throughout this section, $K$ is a discrete valuation field with valuation ring $O _{K}$ and residue field $k$
of characteristic $p \geq 0$.

\begin{lem}
 Let $X$ be a regular scheme of finite type and flat over $O_{K}$, $D \subset  X$ the special fiber, and let us suppose that the closed subset $D$ is an irreducible reduced divisor of $X$. 
Let $G$ be a finite group, and $\phi : Y \rightarrow X \setminus D$ a $G$-torsor ramified at $D$.
Then there exists a closed point $x \in D$ and a $1$-dimensional subspace $L$ of the tangent
space $T_{x}X \overset{\mathrm{def}}{=} (m_{x}/m^2_{x})^{*}$ with the following property:

(C) if $C \subset  X_{\bar{x}}$ is any regular $1$-dimensional closed subscheme tangent to $L$ such that
$C \not\subset D_{\bar{x}}$ then the pullback of $\phi : Y \rightarrow X \setminus D$
to $C \setminus \{ \bar{x} \}$ is ramified at $\bar{x}$.

Here $\bar{x}$ is a geometric point corresponding to $x$ and $X_{\bar{x}}$, $D_{\bar{x}}$ are the strict Henselizations.
\end{lem}
\begin{proof}
 Let $\bar{Y}$ be the normalization of $X$ in the ring of fractions of $Y$, and $\bar{\pi} : \bar{Y} \rightarrow X$
be the canonical morphism. Let us denote the generic point of $D$ by $\xi_{X}$
and choose a ramified point $\xi_{Y}$ in $\bar{Y}$ over $\xi_{X}$.
 
 Let us consider the quotient scheme $\bar{Y}/{I}$ of $\bar{Y}$ by the inertia subgroup $I$ of the decomposition group at $\xi_{Y}$ and the open subscheme $X'$ of $\bar{Y}/{I}$ obtained by removing divisors over $\xi_{X}$ except for the image of $\xi_{Y}$.
 
First we prove the lemma in the case $X = X'$, so we can assume that the fiber $\bar{\pi}^{-1}(\xi_{X})$ is equal to $\{ \xi_{Y}\}$.
 Then $G$ is solvable and so we can assume that $|G|$ is a prime number $q$ by replacing $Y$ by $Y/H$, where $H \subset G$
is a normal subgroup of prime index.

 The extension of the rings of fractions of $Y$ and $X$ is finite separable, so $\bar{\pi}$ is a finite morphism.
Since $Y$ is finite type over $O_{K}$, its regular locus is open by \cite{EGA} Corollaire 6.12.6 and it contains $\xi_{Y}$.
To prove the lemma we can replace $X$ by any open subscheme of $X$ which contains $\xi_{X}$.  Thus we can assume that $Y$ is regular by shrinking $X$ because $\bar{\pi}$ is a closed map. 

Set $\widetilde{D}:= (\bar{\pi}^{-1}(D))_{\mathrm{red}}$.
Since $\widetilde{D}$ is integral and finite type over $O _{K}$, the regular locus of $\widetilde{D}$
is open and we may assume that $\widetilde{D}$ and $D$ are regular.
Then, we can prove the theorem for any closed point $x \in D$.
 The assumption $I = G$ means that the action of $G$
on $\widetilde{D}$ is trivial and the morphism $\bar{\pi}_{D} : \widetilde{D} \rightarrow D$
is purely inseparable. Let $e_{1}$ be its degree and let $e_{2}$ be the multiplicity of $\widetilde{D}$
in the divisor $\pi^{-1}(D)$. Then $e_{1}e_{2} = |G| = q$, so $e_{1}$ equals $1$ or $q$.

 Case 1: $e_{1} = 1$, $e_{2} = q$. Since $e_{1} = 1$, the morphism 
$\bar{\pi}_{D} : \widetilde{D} \rightarrow D$ is an isomorphism. If $L \not\subset T_{x}D$ and $C \subset X_{\bar{x}}$
is any regular 1-dimensional closed subscheme tangent to $L$, $T_{\bar{x}}C$ is transversal to the image of the tangent map
$T_{\bar{\pi}^{-1}(x)}\bar{Y} \rightarrow T_{x}X$. Since the scheme $C \times_{Y} \bar{Y}$ is regular and the set $\bar{\pi}^{-1}(\bar{x})$ is a point, the pullback of
$\pi : Y \rightarrow X \setminus D$ to $C \setminus {\bar{x}}$ is indeed ramified at $\bar{x}$.

 Case 2: $e_{1} = q$, $e_{2} = 1$. Fix any closed point $y \in \widetilde{D}$, and let $x$ be $\bar{\pi}_{D}(y)$.
If the extension of their residue fields $k(y) \subset k(x)$
is nontrivial, it is purely inseparable, so any $L$ satisfies the condition (C).
Let us assume that $k(y) \cong k(x)$. Let us denote the maximal ideal of $O_{X,x}$, $O_{D,x}$, $O_{Y,y}$, $O_{\widetilde{D},y}$ as
$m_{x}$, $n_{x}$, $m_{y}$, $n_{y}$ and choose a local defining equation
$f \in O_{X,x}$ of $D$. Since $e_{2} = 1$, $O_{Y,y} / (f) \cong O_{\widetilde{D},y}$. Then, we have the following commutative diagram with horizontal lines exact:
\[
\xymatrix{
0 \ar[r] & ((f)+m_{x}^{2})/m_{x}^{2} \ar[r] \ar[d]
& m_{x}/m_{x}^{2} \ar[r] \ar[d] & n_{x}/n_{x}^{2} \ar[d] \ar[r] & 0 \\ 
0 \ar[r] & ((f)+m_{y}^{2})/m_{y}^{2} \ar[r] & m_{y}/m_{y}^{2} \ar[r] & n_{y}/n_{y}^{2} \ar[r] & 0.
}
\]
The left vertical arrow is an isomorphism, but $m_{x}/m_{x}^{2} \rightarrow m_{y}/m_{y}^{2}$ is not an isomorphism because
$\bar{\pi}$ is not etale at $y$, so the right vertical arrow is not an isomorphism.
If we choose $L \subset T_{x}X$ with $L \subset T_{x}D$ and $L \not\subset \mathrm{Im}(T_{y}\widetilde{D} \rightarrow T_{x}D)$, we can finish the
argument just as in Case 1 since the subspace $\mathrm{Im}(T_{y}\widetilde{D} \rightarrow T_{x}D)$ of $T_{x}D$ is of codimension $1$.

So we finished the proof in the case $X = X'$.
Finally we prove the lemma in general case. By the lemma for $X'$, we can take a closed point $x' $ in the topological closure of the image of $\xi_{Y}$ in $X'$ and a 1-dimensional subspace $L'$ of $T_{x'}X'$ which satisfy the condition (C) for $X'$.
Since a closed point of the special fiber of $X'$ is a closed point of the special fiber of  $\bar{Y}/I$, if we define $x$ to be the image of $x'$ by the morphism $\bar{Y}/I \rightarrow X$, it is a closed point of $D$.
 Then, also, since $X' \rightarrow X$ is etale by the Zariski-Nagata purity theorem \cite{SGA1} X Theorem 3.4, $T_{x'}X' = T_{x}X \otimes_{k(x)} k(x')$.
 By the argument of taking $L'$ up to the previous paragraph, we see that one can take the 1-dimensional subspace $L'$ of $T_{x'}X'$ in order that it comes from a 1-dimensional subspace $L$ of $T_{x}X$. Then we see that the condition (C) is satisfied for this $x$ and $L$. So we finished the proof of the lemma in general case.

\end{proof}

\begin{cor}
Let $X$ be a scheme smooth over $O _{K}$ and $U \subset X$ the generic fiber.
Let $\mathcal{E}$ a smooth $\mathbb{Q}_{l}$-sheaf on $U$, where $l$ is a prime number not equal to the characteristic of $k$.
 Suppose that $\mathcal{E}$ does not extend to a smooth
$\mathbb{Q}_{l}$-sheaf on $X$. Then there exists a closed point $x \in X \setminus U$ and a 1-dimensional subspace
$L \subset T_{x}X$ with the following property:

(*) if $C$ is a Dedekind scheme of finite type over $O _{K}$, $c \in C$ a closed point such that the extension $k(c) \subset k(x)$ is separable, and $\varphi : (C,c) \rightarrow (X,x)$
a morphism with $\varphi ^{-1}(U) \neq \emptyset$
such that the image of the tangent map $T_{c}C \rightarrow T_{x}X\otimes _{k(x)}k(c)$
equals $L\otimes _{k(x)}k(c)$ then the pullback of $\mathcal{E}$ to $\varphi ^{-1}(U)$
is ramified at $c$.
\label{5el}
\end{cor}

\begin{proof}
By the Zariski-Nagata purity theorem \cite{SGA1} X Theorem 3.4, $\mathcal{E}$ is ramified along some
irreducible divisor $D \subset X$, $D \cap U = \emptyset$. Now use the above lemma.
\end{proof}

\bigskip
 We can also show the above corollary for an ind-smooth $\mathbb{Q}_{l}$-sheaf $\mathcal{E}$ on $U$.
In fact, assume that $\mathcal{E}$ does not extend to an ind-smooth $\mathbb{Q}_{l}$-sheaf on $X$.
We can write $\mathcal{E}$ as inductive system $(\mathcal{E}_{i})_{i \in I}$, where each of its structure morphisms is injective.
Then, there exists $i_{0} \in I$ such that $\mathcal{E}_{i_{0}}$ does not extend to a smooth $\mathbb{Q}_{l}$-sheaf on $X$.

\section{Proof of the main theorem}

\begin{dfn}
Let $S$ be a scheme and $X$ a scheme over $S$.
\begin{enumerate}
\item We shall say that $X$ is a {\it proper polycurve} (of relative dimension $n$) over $S$ if
there exists a positive integer $n$ and a (not necessarily unique) factorization of the structure morphism $X \rightarrow S$
\begin{equation}
X = X_{n} \rightarrow X_{n-1} \rightarrow \ldots \rightarrow X_{1} \rightarrow X_{0} = S
\end{equation}
such that, for each $i \in \{ 1, \ldots ,n \}$, $X_{i} \rightarrow X_{i-1}$ is a proper curve (cf. Definition \ref{eq:3def}).
We refer to the above factorization of $X \rightarrow S$ as a {\it sequence of parametrizing morphisms}.
\item We shall say that $X$ is a {\it proper polycurve with sections} (of relative dimension $n$) over $S$ if
$X$ is proper polycurve (of relative dimension $n$) over $S$, whose parametrizing morphisms
can be chosen so that for all $i \in \{ 1, \ldots ,n \}$, $ X_{i} \rightarrow X_{i-1}$ is a proper curve with section (cf. Definition \ref{eq:3def}).
We refer to such a sequence of parametrizing morphisms as  a {\it sequence of parametrizing morphisms with sections}.
\item Let $X$ be a proper polycurve with sections of relative dimension $n$ over $S$.
For a sequence of parametrizing morphisms with sections of this,
\begin{equation}
\mathcal{S} : X = X_{n} \rightarrow X_{n-1} \rightarrow \ldots \rightarrow X_{1} \rightarrow X_{0} = S,
\end{equation}
we call the maximum of the genera of fibers of $X_{i} \rightarrow X_{i-1}$ the maximal genus $g_{\mathcal{S}}$ of $\mathcal{S}$.
We call the minimum of the maximal genera of sequences of parametrizing 
morphisms with sections of $X$ the maximal genus $g_{X}$ of $X$.
\end{enumerate}
\end{dfn}

Before proving the main theorem, we prove a lemma.

\begin{dfn}
For a profinite group $G$, let us consider the smallest Tannakian subcategory of the category of finite dimensional
continuous $G$-representations over $\mathbb{Q}_{l}$ which contains all the finite dimensional
continuous $G$-representations over $\mathbb{Q}_{l}$ of rank $\leq r$ and which is closed under taking subquotients,
tensor products, duals, and extensions.
We write the Tannaka dual of the Tannakian subcategory with respect to the forgetful functor as $G^{l\text{-alg},r}$.
This definition is compatible to Definition \ref{3tannaka}.3 if $G$ is the etale fundamental group of $X$.
\end{dfn}

\begin{lem}
Let $S$ be a connected Noetherian scheme and take a geometric point $s$ over $S$. Let $r$ be a natural number.
Then, for a prime number $l$ which is invertible on $S$, the morphism
$\pi _{1}(S,s)^{l\text{-alg},r} \rightarrow (\pi _{1}(S,s)^{p'})^{l\text{-alg},r}$
is an isomorphism, if $p$ is $0$ or a prime number which deos not divide any of $l^h-1 (1\leq h\leq r)$.
\label{6char}
\end{lem}
\begin{proof}
 For any continuous representation of $\pi _{1}(S,s)$ over $\mathbb{Q}_{l}$ whose rank is $\leq r$,
the action of $\pi _{1}(S,s)$ factors $\pi _{1}(S,s)^{p'}$ since
$GL(r,\mathbb{Z}_{l}) \cong \underset{\longleftarrow}{\text{lim}} GL(r,\mathbb{Z}/l^{n}\mathbb{Z})$ 
and the order of $GL(r,\mathbb{Z}/l^{n}\mathbb{Z})$ is $l^{(n-1)r^2}(l^{r}-1)(l^{r}-l)\ldots (l^{r}-l^{r-1})$.

 Let $V_{1}, V_{2}$ be finite dimensional continuous representation of $\pi _{1}(S,s)$ over $\mathbb{Q}_{l}$ whose actions of $\pi _{1}(S,s)$ 
factor $\pi _{1}(S,s)^{p'}$. It is easy to see that any subquotient of $V_{1}$, $V_{1}\otimes V_{2}$, and $V_{1}^{\vee}$ are
$\pi _{1}(S,s)^{p'}$-representations. What we should show is for any extension $V$ of $\pi _{1}(S,s)$-representation
\begin{equation}
0 \rightarrow V_{1} \rightarrow V \rightarrow V_{2} \rightarrow 0,
\end{equation}
the action factors $\pi _{1}(S,s)^{p'}$. Thus it is sufficient to show the image $\text{Im}(\pi _{1}(S,s) \rightarrow \text{Aut}V)$
is a pro-$p'$ group. It follows from the fact that the image of the homomorphism
\begin{equation}
\text{Im}(\pi _{1}(S,s) \rightarrow \text{Aut}V) \rightarrow  \text{Aut}V_{1} \times \text{Aut}V_{2}
\end{equation}
is pro-$p'$ and its kernel is pro-$l$.
Therefore $\pi _{1}(S,s)^{l\text{-alg},r}$ and $(\pi _{1}(S,s)^{p'})^{l\text{-alg},r}$ are isomorphic.
\end{proof}

 Let $K$ be a discrete valuation field with valuation ring $O _{K}$ and residue field $k$ of characteristic $p \geq 0$.
Let $X$ be a scheme of finite type geometrically connected over $K$.
For any closed point $x \in X$, let $K(x)$ be the residue field of $x$ and $G_{K(x)}$ the absolute Galois group of $K(x)$.
Choose a valuation ring $O_{K(x)}$ of $K(x)$ over $O_{K}$ and let $I_{K(x)}$ be the inertia subgroup 
of $O_{K(x)}$ (which is well-defined up to conjugation).
Because $x$ is $K(x)$-rational, $\pi_1(X \otimes K(x)^{\rm sep}, \bar{x})^{p'}$ admits 
an action (not an outer action) of $G_{K(x)}$. Note that the triviality of the action of inertia is independent of its choice.

\begin{thm}
Assume that $X$ is a proper polycurve with sections over $K$ and $g = g_{X}$ is the maximal genus of $X$.
Consider the following conditions.
\begin{description}
\item{(A)}\mbox{} $X$ has good reduction.
\item{(B)}\mbox{} The action of $I_{K(x)}$ on
$\pi _1 (X \otimes K(x)^{\mathrm{sep}} ,\bar{x})^{p'}$ is trivial for any closed point $x \in X$, valuation ring $O_{K(x)}$ of $K(x)$ over $O_{K}$,
and geometric point $\bar{x}$ over $\mathrm{Spec} K(x)$.
\end{description}
Then, we have (A) $\Rightarrow$ (B). If we assume that $p = 0$ or $p > 2g+1$, (B) $\Rightarrow$ (A) follows.
\end{thm}

\begin{proof}
From Remark \ref{3hgr}, (A) $\Rightarrow$ (B) follows. Let us prove (B) $\Rightarrow$ (A).
 Fix a sequence of parameterizing morphisms with sections
$$
X = X_{n} \overset{f_{n}}{\underset{s_{n}}{\rightleftarrows}} X_{n-1} \overset{f_{n-1}}{\underset{s_{n-1}}{\rightleftarrows}} \ldots 
\overset{f_{i+1}}{\underset{s_{i+1}}{\rightleftarrows}} X_{i} \overset{f_{i}}{\underset{s_{i}}{\rightleftarrows}} X_{i-1} 
\overset{f_{i-1}}{\underset{s_{i-1}}{\rightleftarrows}} \ldots \overset{f_{2}}{\underset{s_{2}}{\rightleftarrows}} X_{1}
\overset{f_{1}}{\underset{s_{1}}{\rightleftarrows}} X_{0} = \mathrm{Spec}K
$$
of $X \rightarrow \mathrm{Spec}K$, whose maximal genus is $g$.

 We will show that $X_{n}$ has good reduction by induction on $n$.
 For $n = 1$, it is proved in Proposition \ref{3gr}. Now we assume that $n \geq 2$.
For any closed point $y \in X_{n-1}$, the natural surjection
$\pi _1 (X_{n} \otimes K(x)^{\mathrm{sep}},\bar{y})^{p'} \rightarrow \pi _1 (X_{n-1} \otimes K(x)^{\mathrm{sep}},\bar{y})^{p'}$
makes the action of $I_{K(y)}$ on $\pi _1 (X_{n-1} \otimes K(x)^{\mathrm{sep}},\bar{y})^{p'}$ trivial, where $\bar{y}$ is a geometric point above $y$.
Moreover, the above sequence of parameterizing morphism gives the condition of $X_{n-1} \rightarrow \mathrm{Spec}K$ about maximal genus.
Therefore, we may assume that we have a smooth model $\mathfrak{X}_{n-1}$ of $X_{n-1}$.

Fix a prime number $l$ such that $p$ does not divide any of $l^h-1 (1\leq h\leq 2g)$. Note that there exists such a prime number
by the theorem on arithmetic progressions and the assumption $p > 2g+1$.
 For any closed point $x \in X_{n-1}$ and any geometric point $\bar{x}$ over $x$, we have an exact sequence of affine group schemes
\begin{eqnarray}\begin{split}
1 \rightarrow \pi _{1}(X_{n}\times _{X_{n-1}}\bar{x},\bar{x})^{l\text{-unip}}
&\rightarrow \pi _{1}(X_{n}\times _{\mathrm{Spec}K}\bar{x},\bar{x})^{l\text{-rel-unip},2g} \\
&\rightarrow \pi _{1}(X_{n-1}\times _{\mathrm{Spec}K}\bar{x},\bar{x})^{l\text{-alg},2g} \rightarrow 1
\end{split}
\label{eq:hes}
\end{eqnarray}
by Theorem \ref{4hes}.

 Now we consider the diagram
\[
\xymatrix{
X_{n}\times _{X_{n-1}}K(X_{n-1}) \ar[r] \ar[d]^{f'_{n}} & X_{n} \ar[d]^{f_{n}} \ar@{.>}[r] & \mathfrak{X}_{n} \ar@{.>}[d]
& \mathfrak{X} \ar@{.>}[l] \ar@{.>}[d] \\
\mathrm{Spec}\,K(X_{n-1}) \ar[r] \ar@<2.mm>[u]^{s'_{n}} & X_{n-1} \ar[r] \ar@<2.mm>[u]^{s_{n}} & \mathfrak{X}_{n-1}
& \mathrm{Spec}\,O_{\mathfrak{X}_{n-1},\xi} \ar[l]
}
\]
where $K(X_{n-1})$ is the function field of $X_{n-1}$, $\xi$ is the generic point of the special fiber
$\mathfrak{X}_{n-1} \setminus  X_{n-1}$, $f'_{n}$, $s'_{n}$ are the base change of $f_{n}$, $s_{n}$ respectively,
and $O_{\mathfrak{X}_{n-1},\xi}$ is the local ring at $\xi$.
What we should show is that there exist a proper hyperbolic curve $\mathfrak{X}_{n}$ over 
$\mathfrak{X}_{n-1}$ whose base change by $X_{n-1} \rightarrow \mathfrak{X}_{n-1}$ is isomorphic to $X_{n}$.
Thanks to \cite{Mor}, it is sufficient for this to show that there exists a proper hyperbolic curve $\mathfrak{X}$ over
$\mathrm{Spec}\,O_{\mathfrak{X}_{n-1},\xi}$ whose base change by $\mathrm{Spec}\,K(X_{n-1}) \rightarrow \mathrm{Spec}\,O_{\mathfrak{X}_{n-1},\xi}$
is isomorphic to $X_{n}\times _{X_{n-1}}K(X_{n-1})$.
Fix an inertia subgroup $I \subset G_{K(X_{n-1})}$ at $\xi$ and a geometric point $\bar{t}$ of $X_{n}\times _{X_{n-1}}K(X_{n-1})$ above $\mathrm{Spec}\,K(X_{n-1})$.
To complete the proof, it comes down to show that the action of $I$ on $\pi _{1}(X_{n}\times _{X_{n-1}}\bar{t},\bar{t})^{l\text{-unip}}$
is trivial by Proposition \ref{3gr}.

 Let us denote the morphism $\mathrm{Spec}\,K(X_{n-1}) \rightarrow X_{n-1}$ as $i$. Then, we have the following exact sequences
of affine group schemes over $\text{Et}_{l}^{\leq 2g}(K(X_{n-1}))$:
\hspace{-1cm}
\scalebox{0.8}{
\begin{xy}
\xymatrix{
0 \ar[r] & \mathrm{Ker}f'_{n*} \ar[r] \ar[d]
& s_{n}'^{*}\pi (U_{f'_{n}}\text{Et}_{l}^{\leq 2g}(X_{n}\times _{X_{n-1}}K(X_{n-1}))) \ar[r]^(0.6){f_{n*}'} \ar[d]
& \pi (\text{Et}_{l}^{\leq 2g}(\mathrm{Spec}K(X_{n-1}))) \ar[d] \ar[r] & 0 \\ 
0 \ar[r] & i^{*}\mathrm{Ker}f_{n*} \ar[r] & i^{*}s_{n}^{*}\pi (U_{f_{n}}\text{Et}_{l}^{\leq 2g}(X_{n})) \ar[r]^{i^{*}(f_{n*})}
& i^{*}\pi (\text{Et}_{l}^{\leq 2g}(X_{n-1})) \ar[r] & 0.
}
\end{xy}}

By Theorem \ref{4hes}, $\mathrm{Ker}f'_{n*}$ and $i^{*}\mathrm{Ker}f_{n*}$ are isomorphic and their fibers at $\bar{t}$ are isomorphic to
$\pi _{1}(X_{n}\times _{X_{n-1}}\bar{t},\bar{t})^{l\text{-unip}}$. Let us denote the ind-smooth $\mathbb{Q}_{l}$-sheaf on
$X_{n-1}$ corresponding to $\mathrm{Ker}f_{n*}$ by $\mathcal{E}$.  
To show the triviality of the action of $I$ on $\pi _{1}(X_{n}\times _{X_{n-1}}\bar{t},\bar{t})^{l\text{-unip}}$, it is sufficient to prove
that $\mathcal{E}$ extends to $\mathfrak{X}_{n-1}$. To show this, we consider the following claim.

\bigskip
\textbf{Claim}
Let $x$ be a closed point in $\mathfrak{X}_{n-1} \setminus X_{n-1}$ and $O_{L}$ be any discrete valuation ring over $O_{K}$
whose field of fraction $L$ is a finite extension over $K$.
Then, for every morphism $\mathrm{Spec}\,O_{L} \rightarrow \mathfrak{X}_{n-1}$ over $O_{K}$ sending the closed point $y \in \mathrm{Spec}\,O_{L}$
to $x$, $\mathcal{E}|_{\mathrm{Spec}\,L}$ is unramified at $y$.

\bigskip
We prove the claim. Let us define the valuation ring $O_{K(x')}$ of the residue filed $K(x')$ at the image $x'$ of the generic point of $O_{L}$ by $O_{K(x')} = O_{L} \cap K(x')$.
By condition (B) and Lemma \ref{6char}, the action of $I$ on $\pi _{1}(X_{n}\times _{\mathrm{Spec}K}\bar{x'},\bar{x'})^{l\text{-rel-unip},2g}$
is trivial. Therefore for $O_{K(x')}$, this claim follows from (\ref{eq:hes}), and so is for  $O_L$.

Finally, we prove that $\mathcal{E}$ extends to $\mathfrak{X}_{n-1}$.
By the above claim and Corollary \ref {5el}, it is suffices to show that for all closed point $x$ in $\mathfrak{X}_{n-1} \setminus X_{n-1}$ and for all $1$ dimensional linear subspace $M \subset T_{x}\mathfrak{X}_{n-1}$ there exists a discrete valuation ring $O_{L}$ over $O_{K}$ whose field of fraction $L$ is a finite extension over $K$ and $O_{K}$-morphism  $\text{Spec}\, O_{L} \rightarrow \mathfrak{X}_{n-1}$ which induces the isomorphism from the tangent space of the closed point $y$ of $\text{Spec}\, O_{L}$ to $M$. Let us denote  the maximal ideal of the local ring $O_{\mathfrak{X}_{n-1},x}$ as $m_{x}$ and the $n-1$ dimensional linear subspace of $m_{x}/m_{x}^{2}$ annihilated by $M$ as $N$. Let us choose a regular system of parameter   $\{ t_{1} , \ldots , t_{n} \}$ of $m_{x}$ so that the image of $\{ t_{1} , \ldots , t_{n-1} \}$ becomes a basis of $N$.

Case 1 : If the image of the maximal ideal of $O_{K}$ to $m_{x}/m_{x}^{2}$ is not contained in $N$, the quotient ring $O_{\mathfrak{X}_{n-1},x}/(t_{1} , \ldots , t_{n-1})$ works as $O_{L}$. 

Case 2 :  If the image of the maximal ideal of $O_{K}$ to $m_{x}/m_{x}^{2}$ is contained in $N$, we may assume that $t_{1} = \varpi - t_{n}^{2}$ . Here, $\varpi$ is a generator of the maximal ideal of $O_{K}$. Then, the quotient ring $O_{\mathfrak{X}_{n-1},x}/(t_{1} , \ldots , t_{n-1})$ works as $O_{L}$.

\end{proof}

\section{Appendix}
In this section, we give a proof of Proposition \ref{3odatama}, which is not appeared in any published paper. Let us restate the proposition.

Let $S$ be the spectrum of a discrete valuation ring $O_K$, $\eta$ the generic point of $S$,
$s$ the closed point of $S$, $K = \kappa (\eta)$ the fractional field of $O_K$, $k = \kappa (s)$ 
the residue field of $O_K$, and $p$ the characteristic of $k$.

Consider a proper hyperbolic curve $X \rightarrow \mathrm{Spec}\,K$ and
take a geometric point $ \overline{t} $ over $X \otimes K^{\mathrm{sep}}$. 

\begin{prop} (\cite{Oda1}\cite{Oda2}\cite{Tama} section 5)
The following are equivalent.
\begin{enumerate}
\item $X$ has good reduction.
\item The outer action $I_K \rightarrow \mathrm{Out}( \pi _1 (X \otimes K^{\mathrm{sep}} ,\overline{t}) ^{p'})$
defined by (\ref{eq:3oa}) is trivial
\item There exists a prime number $l \neq p$ such that the outer action
$I_K \rightarrow \mathrm{Out}( \pi _1 (X \otimes K^{\mathrm{sep}} ,\overline{t}) ^{l})$ defined by (\ref{eq:3oa}) is trivial.
\item There exists a prime number $l \neq p$ such that the outer action of $I_K$ on 
$\pi _1 (X \otimes K^{\mathrm{sep}},\overline{t}) ^{l}/ \Gamma _{n}
\pi _1 (X \otimes K^{\mathrm{sep}},\overline{t}) ^{l}$ induced by (\ref{eq:3oa})
is trivial for all natural numbers $n$.
\end{enumerate}
\label{newoda}
\end{prop}

We can show 1 $\Rightarrow$ 2 as the proof of Proposition \ref{3gr}, and the assertion that 2  $\Rightarrow$ 3 $\Rightarrow$ 4 is trivial. 
To show 4 $\Rightarrow$ 1, we may assume that the Jacobian of $X$ has good reduction and that $X$ has bad reduction as in \cite{Oda2}.
Then, we have the outer action of the absolute Galois group: 
$$\rho_{l} (\text{mod}\, m) : I_{K} \rightarrow \text{Out}\,\pi _1 (X \otimes K^{\mathrm{sep}},\overline{t}) ^{l}/ \Gamma _{m+1}
\pi _1 (X \otimes K^{\mathrm{sep}},\overline{t}) ^{l}.$$

\begin{thm}
Under the above assumption, the homomorphism
$$\rho_{l} (\text{mod}\, 2) : I_{K} \rightarrow \text{Out}\,\pi _1 (X \otimes K^{\mathrm{sep}},\overline{t}) ^{l}/ \Gamma _{3}
\pi _1 (X \otimes K^{\mathrm{sep}},\overline{t}) ^{l}$$
is trivial, and the homomorphism
$$\rho_{l} (\text{mod}\, 3) : I_{K} \rightarrow \text{Out}\,\pi _1 (X \otimes K^{\mathrm{sep}},\overline{t}) ^{l}/ \Gamma _{4}
\pi _1 (X \otimes K^{\mathrm{sep}},\overline{t}) ^{l}$$
which factors through $I_{K} \rightarrow I^{l}_{K} \rightarrow \text{Out}\,\pi _1 (X \otimes K^{\mathrm{sep}},\overline{t}) ^{l}/ \Gamma _{4}
\pi _1 (X \otimes K^{\mathrm{sep}},\overline{t}) ^{l}$, defines the injective homomorphism,
$$\rho ' _{l} (\text{mod}\, 3) : I_{K}^{l} \rightarrow \text{Out}\,\pi _1 (X \otimes K^{\mathrm{sep}},\overline{t}) ^{l}/ \Gamma _{4}
\pi _1 (X \otimes K^{\mathrm{sep}},\overline{t}) ^{l}.$$
\label{newoda2}
\end{thm}

Let us denote the stable model of $X$ by $\mathfrak{X}$ and suppose that the special fiber has $n$ nodes $x_{1}. \ldots, x_{n}$. To show Theorem \ref{newoda2}, we may assume that $O_{K}$ is a complete discrete valuation ring and $k$ is an algebraically closed field.
Let us construct a two dimensional family of stable curves.
Consider the classifying morphism of $\mathfrak{X}$
$$\text{cl}_{\mathfrak{X}} : \text{Spec}(O_{K}) \rightarrow \mathscr{M}_{g}.$$
Here, $\mathscr{M}_{g}$ is the moduli stack of stable curves of genus $g$ over $\mathbb{Z}_{p}$.
We write the induced morphism $\text{Spec}(k) \rightarrow \mathscr{M}_{g}$ as $\kappa$.
As a regular system of parameters of the strict henselian local ring $\text{Spec}\, O^{\text{sh}}_{\mathscr{M}_{g}, \kappa}$ at $\kappa$, we can choose $3g - 2$ elements $p, T_{1} , \ldots , T_{3g-3}$. We can assume that the local equation of the singularity $\tilde{x}_{i}$ of the universal family $\mathscr{C}_{g}$ over $x_{i}$ is given by 
$$X_{i}Y_{i} = T_{i} \quad (1\leq i \leq n).$$
If we write the local equation of $x_{i}$ as 
$$X'_{i}Y'_{i} = a_{i} \quad (0 \neq a_{i} \in \pi O_{K}),$$ where $\pi$ is a uniformizer of $O_{K}$,
the local homomorphism
$O^{\text{sh}}_{\mathscr{C}_{g}, \tilde{x}_{i}} \rightarrow O^{\text{sh}}_{\mathfrak{X},x_{i}}$
sends $X_{i}$, $Y_{i}$ and $T_{i}$ to $u_{i}X'_{i}$, $u_{i}^{-1}Y'_{i}$ and $a_{i}$ by \cite{Moch3}\S 3.7, \S 3.8 and \cite{Kato} Lemme 2.1, Lemme 2.2. Here, $u_{i}$ is a unit in $O^{\text{sh}}_{\mathfrak{X},x_{i}}.$

Let us consider the ring
$$R = \begin{cases} O_{K}[[t]] & (\text{char} K = 0) \\
    W(k)[[t]] & (\text{char} K = p)
  \end{cases}$$
and the natural morphism $R \rightarrow O_{K}$ which sends $t$ to $\pi$. Here, $W(k)$ is the Witt ring for $k$.

We write the local homomorphism $O_{\tilde{\mathscr{M}_{g}},\tilde{\kappa}}^{\text{sh}} \rightarrow O_{K}$ the induced by the natural morphism $\tilde{\kappa} : \text{Spec}\, k \rightarrow \text{Spec}\, O_{K} \rightarrow \mathscr{M}_{g}\times _{\mathbb{Z}_{p}} \text{Spec}\, R =: \tilde{\mathscr{M}_{g}}$ as $\tilde{\phi}$. 
The elements of $O_{\tilde{\mathscr{M}}_{g},\tilde{\kappa}}^{\text{sh}}$
$$ \begin{cases} \pi, t, T_{1}, \ldots, T_{3g-3} & (\text{char} K = 0) \\
      p, t, T_{1}, \ldots, T_{3g-3} & (\text{char} K = p)
  \end{cases}$$
become a regular system of parameters. 

Choose an element $\tilde{a}_{i} \in R$ so that its image in $O_{K}$ is equal to $a_{i}$ if $1\leq i \leq n$ and the image of $T_{i}$ in $O_{K}$ if $n+1\leq i $, and set $U_{i} = T_{i} -\tilde{a}_{i}$.
The subset of $O_{\tilde{\mathscr{M}}_{g},\tilde{\kappa}}^{\text{sh}}$
$$ \begin{cases} \pi, t, U_{1}, \ldots, U_{3g-3} & (\text{char} K = 0) \\
      p, t, U_{1}, \ldots, U_{3g-3} & (\text{char} K = p)
  \end{cases}$$
becomes a regular system of parameters again.
Then, it holds that $U_{i} \in \text{Ker} (\tilde{\phi})$.
We write the quotient ring $O^{\text{sh}}_{\tilde{\mathscr{M}}_{g}, \tilde{\kappa}}/(U_{1} , \ldots , U_{3g-3})$ as $A$ and the induced homomorphism $A \rightarrow O_{K}$ as $\psi$. If we denote the field of fraction of the strict henselization $A^{\text{sh}}_{(t)}$ of $A_{(t)}$ by $L$, then we get the following commutative diagram.
\[
\xymatrix{
\text{Spec}\, L \ar[r] \ar[d] & \text{Spec}\, A[1/t] \ar[d] & \text{Spec}\, K \ar[d] \ar[l]   \\
\text{Spec}\, A^{\text{sh}}_{(t)} \ar[r] & \text{Spec}\, A & \text{Spec}\, O_{K} \ar[l].
}
\]

Let us recall Abyankar's lemma.

\begin{prop}
Let $\bar{K}, \bar{L}$ be an separable closure of $K, L$.
Then we have the natural outer isomorphisms
$$\text{Gal}(\bar{L}/L)^{p'} \cong \pi_{1}(\text{Spec}\, A[1/t])^{p'} \cong \text{Gal}(\bar{K}/K)^{p'}.$$
\label{aby}
\end{prop}

 Let us start the proof of Theorem \ref{newoda2}.
Let us write the local monodromy homomorphism as
$$\rho : \text{Gal}(\bar{K}/K) \rightarrow \text{Out}\, \pi_{1}(X\otimes \bar{K},\bar{t})^{l} .$$
Since $\text{Gal}(\bar{K}/K)$ acts trivially on the abelianization $\pi_{1}(X\otimes \bar{K},\bar{t})^{l,\text{ab}}$, $\rho$ induces the homomorphism $\text{Gal}(\bar{K}/K)^{l} \rightarrow \text{Out}\,\pi_{1}(X\otimes \bar{K},\bar{t})^{l}.$
 By Proposition \ref{aby}, we have the morphism
$$\pi_{1}(\text{Spec}\, A[1/t]) \rightarrow \pi_{1}(\text{Spec}\, A[1/t])^{l} \rightarrow \text{Out}\, \pi_{1}(X\otimes \bar{K},\bar{t})^{l},$$
which is compatible with the isomorphisms of Proposition \ref{aby} and the monodoromy
homomorphism
$$\rho_{0}:\text{Gal}(\bar{L}/L) \rightarrow \text{Out}\,\pi_{1}(\mathfrak{Y}\otimes \bar{L},\xi).$$
Here, we write the restriction of the curve $\mathscr{C}_{g}$ to the scheme $\text{Spec}\, A$ as $\mathfrak{Y}$ and a geometric point of the curve $\mathfrak{Y}\otimes \bar{L}$ as $\xi$. It hold that $\tilde{a}_{i}$ is not equal to $0$ in $L$, so $\mathfrak{Y}\otimes L$ is smooth over $\text{Spec}\, L$.
Since the residue characteristic of $\text{Spec}\, A^{\text{sh}}_{(t)}$ is $0$, this theorem follows from the transcendental method in [Oda2].


\begin{thebibliography}{99}
\bibitem[Bo]{Bo} N. Bourbaki, {\it Elements de mathematique. Groupes et algebres de Lie. Chapitre II: Algebres de Lie libres. Chapitre III: Groupes de Lie}, Actualites Scientifiques et Industrielles, No. 1349. Hermann, Paris, 1972.
\bibitem[Del]{Del} P. Deligne, {\it Le groupe fondamental de la droite projective moins trois points}, Galois Groups over $\mathbb{Q}$, Math. Sci. Res. Inst. Pub., 1989, pp. 79-297.
\bibitem[Dri]{Dri} V. Drinfeld, {\it On a conjecture of Deligne}, Mosc. Math. J. 12 (2012), no. 3, 515-542, 668. 
\bibitem[EGA]{EGA} A. Grothendieck, {\it El\'ements de g\'eom\'etrie alg\'ebrique. IV. \'Etude locale des sch\'emas et des morphismes de sch\'emas. II}, Inst .Hautes \'Etudes Sci. Publ. Math. No. 24 (1965).
\bibitem[Ho]{Ho} Y. Hoshi, {\it The Grothendieck conjecture for hyperbolic polycurves of lower dimension}, to appear in J. Math. Sci. Univ. Tokyo.
\bibitem[Kato]{Kato} F. Kato, {\it Log smooth deformation and moduli of log smooth curves.}, Internat. J. Math. 11 (2000), no. 2, 215-232.
\bibitem[Lab]{Lab} J.-P. Labute, {\it On the descending central series of groups with a single defining relation}, J. Algebra 14 1970 16-23.
\bibitem[Laz]{Laz} C. Lazda, {\it Relative fundamental groups and rational points}, To appear in Rend. Semin. Mat. Univ. Padova
\bibitem[Moch1]{Moch1} S. Mochizuki, {\it The local pro-$p$ anabelian geometry of curves}, Invent. Math. 138 (1999), no. 2, 319-423.
\bibitem[Moch2]{Moch2} S. Mochizuki, {\it Absolute anabelian cuspidalizations of proper hyperbolic curves}, J. Math. Kyoto Univ. 47 (2007), no. 3, 451-539.
\bibitem[Moch3]{Moch3} S. Mochizuki, {\it The geometry of the compactification of the Hurwitz scheme}, Publ. of RIMS. 31. 1995. 355-411.
\bibitem[Mor]{Mor} L. Moret-Bailly, {\it Un theoreme de purete pour les familles de courbes lisses}, C. R. Acad. Sci. Paris, 300 no 14 (1985), 489-492.
\bibitem[Oda1]{Oda1} T. Oda, {\it A note on ramification of the Galois representation on the fundamental group of an algebraic curve}, J. Number Theory, 34 (1990) 225-228.
\bibitem[Oda2]{Oda2} T. Oda, {\it A note on ramification of the Galois representation on the fundamental group of an algebraic curve II}, J. Number Theory, 53 (1995) 342-355.
\bibitem[SGA1]{SGA1} A. Grothendieck, and Mme. Raynaud. M, S\'eminaire de G\'eometrie Alg\'ebrique du Bois Marie 1960/61, {\it Rev\'etements Etales et Groupe Fondamental (SGA 1)}, Lecture Notes in Mathematics, 224, Springer-Verlag, Berlin/Heidelberg/New York, 1971.
\bibitem[ST]{ST} J.-P. Serre, ; J. Tate, {\it Good reduction of abelian varieties}, Ann. of Math. (2) 88 1968 492-517.14.51
\bibitem[Tama]{Tama} A. Tamagawa, {\it The Grothendieck conjecture for affine curves}, Compositio Math. 109 (1997), no.2, 135-194.
\bibitem[Wil]{Wil} J. Wildeshaus, {\it Realizations of polylogarithms}, Lecture Notes in Mathematics, 1650. Springer-Verlag, Berlin, 1997.
\end{thebibliography}
\end{document}